\documentclass{amsart}
\usepackage{latexsym, amsfonts, amsthm, amsmath, amssymb, amscd, epsfig, amsrefs}
\usepackage[all]{xy}
\usepackage{color}
\usepackage{mathrsfs}
\allowdisplaybreaks[1]

\newtheorem{cor}{Corollary}
\newtheorem{lemma}{Lemma}[section]
\newtheorem{Add in proof}{Add in proof}[section]

\newtheorem{remark}{Remark}
\newtheorem{theorem}{Theorem}

\newcommand{\R}{\ensuremath{\mathbb{R}}}
\def \p{\partial}
\def\intave#1{-\kern-10.7pt\int_{\,#1}}
\def\<{\langle}   \def\>{\rangle}
\def\({\left(}    \def\){\right)}

\def\limsup{\operatornamewithlimits{lim\,sup}}

\let\d\relax\DeclareMathOperator{\d}{d}
\renewcommand{\d}{\,d}
\DeclareMathOperator{\curl}{curl}
\DeclareMathOperator{\tr}{tr}

\newcommand{\sign}[1]{\ensuremath\mathrm{sign}(#1)}

\let\div\relax \DeclareMathOperator{\div}{div}

\newcommand{\na}{\ensuremath{\nabla}}

\newcommand\alabel[1]{\addtocounter{equation}{1}\tag{\theequation}\label{#1}}

\begin{document}
\title{A new representation for the Landau-de Gennes energy of nematic liquid crystals}

\numberwithin{equation}{section}
\author{Zhewen Feng and Min-Chun Hong}

\address{Zhewen Feng, Department of Mathematics, The University of Queensland\\
	Brisbane, QLD 4072, Australia}
\email{z.feng@uq.edu.au}

\address{Min-Chun Hong, Department of Mathematics, The University of Queensland\\
Brisbane, QLD 4072, Australia}
\email{hong@maths.uq.edu.au}

\begin{abstract} In the Landau-de Gennes  theory on nematic liquid crystals, the well-known Landau-de Gennes energy depends on  four  elastic constants; $L_1$, $L_2$, $L_3$, $L_4$. For the general case of $L_4\neq 0$, Ball-Majumdar \cite {BM} found an  example that  the   Landau-de Gennes energy functional from physics literature \cite{MN} does not  satisfy a  coercivity condition, which causes a  problem in mathematics to establish  existence of energy minimizers. In order to solve this problem, we observe that the original third order  term  on  $L_4$, proposed by Schiele and Trimper  \cite{ST} in physics, is a linear combination of a fourth order term and a second order term. Therefore, we can propose a new Landau-de Gennes energy, which is equal to the  original for uniaxial nematic $Q$-tensors. The new  Landau-de Gennes energy with  general elastic constants satisfies the coercivity condition for all $Q$-tensors, which establishes a new link between mathematical and physical theory. Similarly to the work of Majumdar-Zarnescu  \cite{MZ}, we prove existence and convergence of minimizers of the new Landau-de Gennes energy. Moreover, we find a new way to study the limiting problem of the Landau-de Gennes system since the cross product method \cite{Chen} on the Ginzburg-Landau equation does not work for the Landau-de Gennes system.
     \end{abstract}

 \maketitle

\pagestyle{myheadings} \markright {A new representation for the Landau-de Gennes energy}

\section{Introduction}

    Liquid crystal is a state of matter between isotropic liquid and crystalline solid.
    Based on molecular positional and orientational orders,  there are three main phases: sematic, cholesterics and  nematic \cite{Mi}*{p.~578}.  The nematic phase is the most common type in  which the general states are  the  uniaxial and   biaxial state.
    Due to the anisotropic microstructure, some physical properties such as light polarization, of substances will change under external influence. It is best known for the use in liquid crystal displays.

 In their  pioneering works, Oseen \cite{Os}  and Frank \cite{Fr}  discovered  the first mathematical  continuum theory of uniaxial nematic liquid crystals through a vector representation. Let $\Omega$ be a domain in $\R^3$. For a unit director $u\in W^{1,2} (\Omega; S^2)$, the  Oseen-Frank energy density  is given by
		 \begin{align}\label{OF density}
			W(u,\na u) =& \frac {k_1} 2(\div u)^2+ \frac {k_2 } 2(u\cdot \curl u)^2+\frac {k_3} 2|u\times \curl u|^2\\
&+\frac {k_2+k_4 } 2(\tr(\na u)^2-(\div u)^2),\nonumber
		\end{align}
where $k_1,k_2,k_3$ are the Frank constants for molecular distortion  of splay, twist and bend  respectively and $k_4$ is the Frank constant for the surface energy.

The Oseen-Frank energy, which can only account  for  uniaxial phases, is one of the successful theories for  modelling nematic liquid crystals in physics \cite{Ste}.
It is also of great interest to study the biaxial phase.
In 1970,  Freiser  \cite {Fre} hypothesized  a rare substance having  a biaxial phase, which was later discovered  by Madsen et al. \cite{MDNS} in 2004. To  study  the phenomenon of phase transitions,  de Gennes \cite{De} in 1971 discovered a matrix representation, known as the $Q$-tensor order parameter,   and the first  expression of the elastic energy of this $Q$-tensor  with  the Landau theory  \cite{SV12}*{p.~208}.  Presently,   the Landau-de Gennes theory  is well-known  for capturing the phase transitions and biaxial state of liquid crystals.
     The  Landau-de Gennes theory  has been verified  in physics  as one of the successful theories for  modelling the nematic liquid crystals. Indeed,
   Pierre-Gilles de Gennes was awarded  a Nobel prize for physics in 1991  ``{\it for discovering that methods developed for studying order phenomena in simple systems can be generalized to more complex forms of matter, in particular to liquid crystals and polymers}''.

 In the Landau-de Gennes framework,  a symmetric, traceless $3\times3$ matrix $Q\in\mathbb M^{3\times3}$ is  known as the $Q$-tensor order parameter, where  $\mathbb M^{3\times 3}$ denotes the space of $3\times 3$ matrices.  The space of symmetric, traceless  $Q$-tensors is defined by
		\begin{equation}\label{tensor space}
		S_0:=\left\{Q\in \mathbb M^{3\times 3}:\quad Q^T=Q, \, \mbox{tr }Q  =0\right\}.
		\end{equation}
For a tensor $Q\in W^{1,2}(\Omega; S_0)$, its Landau-de Gennes energy is defined by
\begin{equation*}
E_{LG}(Q; \Omega)=\int_{\Omega}( f_E + f_B)\,dx,
\end{equation*}
where $f_E$ is the elastic energy density  with elastic constants $L_1,...,L_4$ of the form
	\begin{equation}\label{LG}
f_E(Q,\nabla Q):=\frac{L_1}{2}|\nabla  Q|^2+\frac{L_2}{2} \frac{\p Q_{ij}}{\p x_j} \frac{\p Q_{ik}}{\p x_k} +\frac{L_3}{2}\frac{\p Q_{ik}}{\p x_j}\frac{\p Q_{ij}}{\p x_k}+\frac{L_4}{2}Q_{lk}\frac{\p Q_{ij}}{\p x_l}\frac{\p Q_{ij}}{\p x_k}
\end{equation}
 and $f_B(Q)$ is  a bulk  energy density with three positive constant $a$, $b$, $c$ defined by
\begin{equation}\label{BE}	
f_B(Q):=-\frac{ a}{2}\tr( Q^2)-\frac{ b}{3}\tr( Q^3)+\frac{ c}{4}\left[\tr( Q^2)\right]^2.
\end{equation}
Here and in the sequel, we adopt  the Einstein summation convention for repeated indices.

For a tensor $Q\in W^{1,2}(\Omega; S_0)$, de Gennes \cite{De} first discovered a two-term expression of the elastic energy density in (\ref{LG})
\[\frac{L_1}{2}|\nabla  Q|^2+\frac{L_2}{2}\frac{\p Q_{ij}}{\p x_j} \frac{\p Q_{ik}}{\p x_k}.\]
In 1983, Schiele and Trimper  \cite{ST}*{p.~268} revealed that the
early attempt of de Gennes's work \cite{De} was incomplete since the connection
with the Oseen-Frank density in (\ref {OF density}) would require the splay and bend Frank constants to be equal (i.e. $k_1=k_3$), but, some experiments on liquid crystals showed that $k_3>k_1$, so they extended the original de Gennes representation to one with a third order  term   involving an elastic constant $L_4$:
    \[ \frac{L_1}{2}|\nabla  Q|^2+\frac{L_2}{2} \frac{\p Q_{ij}}{\p x_j} \frac{\p Q_{ik}}{\p x_k} +\frac{L_4}{2} Q_{lk}\frac{\p Q_{ij}}{\p x_l}\frac{\p Q_{ij}}{\p x_k}.\]
    In 1984, Berreman and Meiboom \cite{BM84} observed that above two groups
     discarded the surface energy density in the Oseen-Frank density,
     which correlates the blue phase theory for liquid crystals, so
     they proposed to recover a second order term in $Q$ with four third order terms.
     In 1987, Longa et al. \cite{LMT} gave a extension of Landau-de Gennes density with 22 independent parameters, but
    it is very complicated.
   Later,  Mori et al. \cite{MGKB} in 1999 addressed
 that Dickmann in his PhD thesis \cite{Di} derived a  four
 independent parameters Landau-de Gennes density (\ref{LG}), which is consistent with
 the Oseen-Frank density in (\ref {OF density}).
  Since then, the general form (\ref{LG}) of the Landau-de Gennes representation  has been widely accepted for modelling liquid crystals  (e.g.  \cite {MGKB},  \cite {MN},  \cite {Ba}, \cite  {BZ}).

 From a mathematical point of view, a general form of the tensor $Q\in S_0$ can be written as
	\[Q:=s(u\otimes u-\frac 13 I)+r(w\otimes w-\frac 13 I),\quad u,w\in S^2, \quad s,r\in \R.\]
	Here $u,w$ are two independent direction fields for biaxial liquid crystals and $I$ is the identity matrix.
When the tensor $Q$ has two equal non-zero eigenvalues, a nematic liquid crystal is said to be uniaxial. When $Q$ has two unequal non-zero eigenvalues, a nematic liquid crystal is said to be  biaxial.  For material constants $a , b, c> 0$,  we define the constant order parameter
	\[s_+:=\frac{b+\sqrt{b^2+24ac}}{4c}.\]
We define a subspace
	\[S_*:=\left\{Q  \in S_0:\quad Q =s_+ (u\otimes u-\frac 13 I),\quad  u\in S^2 \right\}.\]	
It is well-known (e.g. \cite {MN}) that $Q\in S_*$ if only if $ \tilde f_B(Q):=f_B(Q)-\inf_{S_0} f_B =0$.	

 	Although there are  many differences between the Oseen-Frank theory and the Landau-de Genes theory, it is of great interest in mathematics and physics whether the Oseen-Frank system can be approximated by the Landau-de Genes system \cite{NZ}.  As it was pointed out in \cite  {MGKB}, Dickmann
 discovered that    for  an  uniaxial  phase  $Q =s(u\otimes u-\frac 13 I)$, the elastic energy density $f_E(Q,\nabla Q)$ in (\ref{LG}) is equal  to the  Oseen-Frank energy density $W(u,\na u)$.
 For the case of uniaxial phase, both  the Oseen-Frank theory and the Landau-de Gennes theory unify in physics modelling.
 In mathematics literature, most  research   focus on the study of the one-constant approximation \cite{Ba};  i.e.  the  elastic parameters satisfy $L_2=L_3=L_4=0$ in (\ref{LG}). Then the density $f_E(Q,\na Q) =\frac {L_1} 2 |\nabla Q|^2$.
In this case, the  Landau-de Gennes energy of $Q\in W^{1,2}(\Omega ; S_0)$ is simplified by
 \begin{equation}\label{SLG}
 E_{SLG}(Q; \Omega)=\int_{\Omega}\left ( \frac  {L_1} 2  |\nabla Q|^2+   f_B(Q)\right )\,dx.
 \end{equation}
Define $W^{1,2}_{Q_0}(\Omega ; S_0)$ be the space $W^{1,2}(\Omega ; S_0)$ with boundary condition $Q_0\in W^{1,2}(\p \Omega ; S_*)$, there is a minimizer of $E_{SLG}$ in $W^{1,2}_{Q_0}(\Omega ; S_0)$, which satisfies the Euler-Lagrange equation
		 \begin{equation}\label{SEL}
		 \Delta Q_{ij}=\frac 1 {L_1} \(-aQ_{ij}-b\Big(Q_{ik}Q_{kj}-\frac{\delta_{ij}}{3}\tr(Q^2)\Big)+cQ_{ij}\tr(Q^2)\).
		 \end{equation}
Majumdar-Zarnescu \cite {MZ} proved that as $L_1\to 0$, minimizers $Q_{L_1}$ of $E_{SLG}$ converges to  $Q_*=s_+ (u^*\otimes u^*-\frac 13 I)$,  where $Q_*$ is a minimizer of $E_{SLG}$ in $W^{1,2}_{Q_0}(\Omega ; S_*)$.   Later,
	 Nguyen-Zarnescu \cite {NZ} improved the result that minimizers $Q_{L_1}$ converge smoothly to $Q_*$ except a singular set.

 In theory of liquid crystals, the general expectation on the elastic constants is that $L_1>0,L_2>0,L_3$ and $L_4$ are not always zero (c.f. \cite{ST}*{p.~268}, \cite{BM}). Therefore,  it  is   very important to study whether the limit of solutions to the Landau-de Gennes  system    is a solution to the  Oseen-Frank system for a general case of $L_1, \cdots , L_4$. In 2D,  Bauman, Park and Phillips \cite{BPP}  investigated a limiting result of minimizers of the energy $E_{LG}$  with $L_4=0$ (see also  \cite{GM}).   For $L_4\neq 0$, Iyer, Xu and Zarnescu \cite{IXZ} studied the 2D problem and imposed a small  condition on the supremum of the unknown $Q$  to gain some control on the $L_4$ term. However, the limiting problem is open for the general case with $L_4\neq 0$.

A fundamental  problem in mathematics on the Landau-de Genes theory  is  to establish existence of minimizers of the energy functional $E_{LG}$ in  $W^{1,2}_{Q_0}(\Omega ; S_0)$  for a general case of elastic constants $L_1,\cdots,L_4$.
To prove the  existence of a minimizer of the functional  $E_{LG}(Q,\Omega)$ in $W^{1,2}(\Omega; S_0)$, one must show that the functional  $E_{LG}$ is lower semi-continuous in $W^{1,2}(\Omega; S_0)$.  By  the standard theory of calculus variations (e.g.  \cite{Gi}), it is necessary to establish  that $f_E(Q, \nabla Q)$ is bounded below by $a|\nabla Q|^2-C$ for some $a>0$. Therefore, it is   very important  to study the bound below problem of $f_{E}(Q,\nabla Q)$.
When $L_4=0$,  Longa et al. \cite{LMT} found the stability criteria
\begin{align}\label{L iff}L_1+L_3>0,\,  2L_1-L_3>0,\, L_1+\frac53L_2+\frac16L_3>0.\end{align}
Under this condition, Davis and Gartland \cite{DG} showed that   $f_E$
 satisfies the coercivity condition.   Kitavtsev et al. \cite{KRSZ}  proved that the condition
 \eqref{L iff}
 is also necessary.
   For the case of $L_4\neq0$  in (\ref{LG}), Ball-Majumdar \cite {BM} found an  example that $f_E(Q, \nabla Q)$ is unbounded from below, so one cannot prove existence of a minimizer of the functional  $E_{LG}(Q, \Omega)$ in $W^{1,2}(\Omega; S_0)$. Therefore, the  Dickmann's representation  (\ref{LG}) causes a knowledge gap between mathematics and physics, which is very challenging in mathematics since the energy functional  $E_{LG}$ in $W^{1,2}(\Omega; S_0)$  does not satisfy a coercivity condition and violates the existence  theorem of minimizers  \cite{Ba}. To attain the coercivity for the case of $L_4\neq 0$,
    Mucci and Nicolodi \cite{MN16}   proved that the energy functional satisfied a coercivity condition under some special conditions on the material constants.
   In contrast to the above continuum theory, Ball and Majumdar \cites{BM} suggested a statistical approach from the Maier-Saupe theory and proposed a singular bulk potential instead of the Landau-de Gennes bulk potential to attain the coercivity condition. This new setting has been investigated by many  \cites{EKT,FRSZ,FRSZ15,Wi,WXZ}. A comprehensive review of this statistical approach,  please refer to \cites{Ba,Ga}.

\medskip

To solve  the above coercivity problem,  we observe  in Lemma \ref{Lemma 2.1} that for uniaxial tensors $Q\in S_*$,
the original
third order  term  on  $L_4$ in  (\ref {LG}), proposed by
  Schiele and Trimper  \cite{ST}*{p.~268} in physics, is a linear combination of a fourth order term and a second order term in the following:
 \begin{align*}\alabel{third}
	Q_{lk}\frac{\p Q_{ij}}{\p x_l}\frac{\p Q_{ij}}{\p x_k}=\frac{3}{s_+}(Q_{ln}\frac{\p Q_{ij}}{\p x_l})(Q_{kn}\frac{\p Q_{ij}}{\p x_k})-\frac {2s_+}3|\na Q|^2.
	\end{align*}
We emphasise that
 the third order  term on $L_4$   by  Schiele and    Trimper  \cite{ST}  depends on
the splay and bend  constants
  $k_1, k_3$; i.e. in general, $L_4 =\frac1{2s_+^{3}} (k_3-k_1)$ is not  zero.
Longa et al. \cite{LMT}   mentioned
 that  the third order  term by
  Schiele and    Trimper
 is a linear combination of six third order terms and also suggested that the  Oseen-Frank energy density is a   linear combination
 of  their 22 independent parameters, but they did not give an explicit  form  such that the energy density satisfies  the above coercivity problem. Our  fourth order term in \eqref{third},  derived from the third order  term  of
  Schiele and    Trimper, is a positive square term and  a linear combination
 of  three fourth order terms $L^{(4)}_5, L^{(4)}_6, L_7^{(4)}$ in \cite{LMT}; i.e. we verify in Lemma \ref{Lemma 2.1.b} that
 \[Q_{ln}Q_{kn}\frac{\p Q_{ij}}{\p x_l}\frac{\p Q_{ij}}{\p x_k}=\frac{8}{5} L^{(4)}_5- \frac{2}{5}L^{(4)}_6 +\frac{2}{5}L_7^{(4)}.\]

  Due to  \eqref{third}, for uniaxial tensors $Q\in S_*$, the elastic energy density  $f_E(Q,\nabla Q)$ in (\ref {LG}) is equivalent to the new form
	\begin{align}\label{D}
	f_{E,1}(Q,\nabla Q)=& \(\frac{L_1}{2}-\frac{s_+L_4}{3}\)|\na Q|^2+\frac{L_2}{2}
\frac{\p Q_{ij}}{\p x_j} \frac{\p Q_{ik}}{\p x_k}
\\&+ \frac{L_3}{2} \frac{\p Q_{ik}}{\p x_j}\frac{\p Q_{ij}}{\p x_k}
+ \frac{3L_4}{2s_+}Q_{ln} Q_{kn}\frac{\p Q_{ij}}{\p x_l}\frac{\p Q_{ij}}{\p x_k}.\nonumber\end{align}
Assuming that
  \begin{align}\label{L cond}L_2\geq 0,\, L_4\geq 0,\, L_1- |L_3|-\frac{2s_+}{3}L_4>0,\end{align}    the energy density $f_{E,1}(Q,\nabla Q)$ in \eqref{D} is rewritten as
\begin{align}\label{f a V}
f_{E,1}=&\frac12(L_1-|L_3|-\frac{2s_+}{3}L_4)|\na Q|^2+ V(Q,\nabla Q),
\end{align}
where
\begin{align*}
V(Q,\nabla Q):=&\frac{L_2}{2}\sum^3_{i=1}\(\sum^3_{j=1}\frac{\p Q_{ij}}{\p x_j}\)^2
		+\frac{|L_3|}{4}\sum^3_{i,j,k=1}\(\frac{\p Q_{ik}}{\p x_j}- \mbox{sign} (L_3)\frac{\p Q_{ij}}{\p x_k}\)^2\\
&
		+\frac{3L_4}{2s_+}\sum^3_{i,j,n=1}\(\sum^3_{k=1}Q_{kn} \frac{\p Q_{ij}}{\p x_k}\)^2.
\end{align*}

 By the new form  of $f_{E,1}(Q,\nabla Q)$ in (\ref{D}),  for each $Q\in W^{1,2}(\Omega, S_0)$, we suggest a new Landau-de Gennes  energy functional
\begin{equation}\label{LG1}
 E_{L}(Q; \Omega )=\int_{\Omega}\(  f_{E,1}(Q, \nabla Q) +\frac 1 L f_B(Q)\)\,dx.
 \end{equation}Here $L$ is a parameter to drive all elastic constants to zero \cites{GM,MN,BPP}.
 Then we have
	\begin{theorem}\label{Theorem 1} Assume that $L_2\geq 0$, $L_4\geq 0$ and $L_1-|L_3|-\frac{2s_+}{3}L_4>0$.
For each $L>0$, there exists a minimizer $Q_L$ of the  new Landau-de Gennes  energy  (\ref{LG1}) in $W^{1,2}_{Q_0}(\Omega; S_0 )$ with a given boundary $Q_0\in W^{1,2}(\Omega; S_* )$.
As $L\to 0$, the minimizers $Q_L$ of   $E_L$ in $W^{1,2}_{Q_0}(\Omega; S_0 )$
		converge strongly  to  $Q_*$ in $W^{1,2}_{Q_0}(\Omega; S_0)$, where $Q_*=s_+(u_*\otimes u_*-\frac 1 3 I)$ is a minimizer of  the elastic energy functional
\[E(Q; \Omega)=\int_{\Omega}   f_{E,1}(Q, \nabla Q)\,dx=\int_{\Omega}   f_{E}(Q, \nabla Q)\,dx\] in $W^{1,2}_{Q_0}(\Omega; S_* )$.  Moreover, $Q_*$ is partially regular in $\Omega$.
\end{theorem}
\begin{remark}When $L_4\geq 0$, using the  result of  Kitavtsev et al. \cite{KRSZ}, $f_{E,1}$  in (\ref{D}) satisfies a coercivity condition  if and only if the constants $L_1,\cdots, L_4$ satisfy the following:
\begin{align}\label{L}
&  L_1+L_3-\frac{s_+}{6}L_4>0,\quad  2L_1-L_3-\frac{s_+}{3}L_4>0,\\
&L_1+\frac53L_2+\frac16L_3-\frac{s_+}{6}L_4>0. \nonumber
\end{align}
Theorem \ref {Theorem 1} holds if $L_1,\cdots, L_4$ satisfy \eqref{L}.
\end{remark}

In Lemma  \ref{Lemma 2.2},  we prove that a minimizer $Q_*$ of $E(Q; \Omega)$ in $W^{1,2}_{Q_0}(\Omega; S_* )$
  satisfies the following Euler-Lagrange equation

		\begin{align*}\alabel{EL}
		&\bar \alpha \(-s_+ \Delta Q_{ij} +2 \na_kQ_{il}\na_kQ_{jl} -2 s_+^{-1}(Q_{ij}+\frac {s_+}3\delta_{ij})|\na Q|^2\)\\
		&-\nabla_k \(
		(Q_{jl}+\frac {s_+}3 \delta_{jl})V_{p^k_{il}}
		+(Q_{il}+\frac {s_+}3 \delta_{il})V_{p^k_{jl}}
		-2s_+^{-1} (Q_{ij}+\frac {s_+}3 \delta_{ij})(Q_{lm}+\frac {s_+}3 \delta_{lm})V_{p^k_{lm}} \)
		\\
		& + V_{p^k_{il}}\na_kQ_{jl}
		+ V_{p^k_{jl}}\na_k Q_{il}
		-2s_+^{-1}V_{p^k_{lm}}\( ( Q_{ij}+\frac {s_+}3 \delta_{ij})\nabla_k Q_{lm}+( Q_{lm}+\frac {s_+}3 \delta_{lm})\nabla_k Q_{ij}\)
		\\
		&+V_{Q_{il}}(Q_{jl}+\frac {s_+}3 \delta_{jl})
		+V_{Q_{jl}}(Q_{il}+\frac {s_+}3 \delta_{il})
		-2s_+^{-1}V_{Q_{lm}}(Q_{lm}+\frac {s_+}3 \delta_{lm}) (Q_{ij}+\frac {s_+}3 \delta_{ij}) =0
		\end{align*}
in the weak sense for $\bar \alpha :=L_1-|L_3|-\frac{2s_+}{3}L_4>0$.
		 In particular, for the case of  $L_2=L_3=L_4=0$,  \eqref{EL} is simplified to
\[s_+\Delta Q_{ij}-2 \na_kQ_{il}\na_kQ_{jl}+2s_+^{-1}(Q_{ij}+\frac {s_+}3\delta_{ij})|\na Q|^2=0, \]
which is equivalent to the harmonic map equation of $u$. Comparing with the result in \cite {Ho2}, the weak solution of (\ref{EL}) might be not unique.

 \medskip
\begin{remark}
When $L_2=L_3=L_4=0$, Majumdar-Zarnescu  \cite {MZ}  proved a monotonicity
formula for minimizers $Q_L$  of $E_{SLG}(Q; \Omega)$ in $W^{1,2}(\Omega, S_0)$. For the case of $L_4=0$, Contreras and Lamy \cite{CL18} proved uniform convergence of $Q_L$ outside of a singular set. However, in general cases of  $L_4\neq 0$, there is no monotonicity
formula for minimizers $Q_L$ of $E_{LG}(Q; \Omega)$ in $W^{1,2}(\Omega, S_0)$, so it is a very interesting question whether  one can improve the convergence of $Q_L$ for   general cases.
\end{remark}

In Theorem \ref {Theorem 1}, we assume    that $L_4\geq 0$. For general case of $L_4$,  we will obtain a new form of the Landau-de Gennes  energy density through a strong Ericksen's condition on the Oseen-Frank  density.  More precisely, using the condition that
\begin{align}\label{RL}
 &s_+^2L_1 =-\frac 16k_1+\frac 12k_2+\frac 16k_3,&
	s_+^2L_2 =&k_1-k_2-k_4,
	\\
	&s_+^2L_3 =k_4,&
	s_+^{3}L_4 =&-\frac12 k_1+\frac12 k_3,\nonumber
\end{align}	
it was shown in \cite  {MGKB} that for each $Q =s_+ (u\otimes u-\frac 13 I)\in S_*$,
\[W(u,\na u) =f_E(Q,\nabla Q). \]
Recent studies \cites{Ba,HM,FHM} revealed that the strong Ericksen  condition on $k_1, \cdots , k_4$ is required for the Oseen-Frank energy to ensure the existence of  minimizers.
Note that $W(u,\na u)$ in (\ref{OF density}) is quadratic in $\na u$, but the $(k_2+k_4)$ term could be negative, so the coercivity $W(u,\na u)\geq a|\na u|^2$ is unclear.  It was pointed out in \cite{HM} (see also \cite{Er2}) that assuming   the strong Ericksen condition
\begin{equation}\label{Er}	k_2>|k_4|,\quad k_3>0,\quad 2k_1>k_2+k_4,
\end{equation}
there are positive constants $\lambda$ and $C$ such that the density $W(u,\na u)$ is equivalent to a form that  $W(z,p)$  satisfies
\begin{equation*}
 \lambda |p|^2\leq W(z,p)\leq C|p|^2,\quad
 \lambda |\xi|^2\leq W_{p_i^kp_j^l}(z,p) \xi_i^k \xi_j^l\leq C |\xi|^2
\end{equation*}
for any $\xi\in\mathbb M^{3\times 3}$, any $p\in\mathbb M^{3\times 3}$ and any $z\in \R^3$ with $|z|\leq M$ for some constant $M>0$ (see details in Lemma \ref{Lemma 3.1}).

Through the relation (\ref{RL}) between Frank's consists $k_1, \cdots, k_4$ and elastic constants $L_1, \cdots, L_4$, the strong Ericksen condition (\ref {Er})
 is equivalent to a condition  that
 \begin{align}\label{Er1}	
	 L_1-\frac{1}{2}|L_3|& > \frac{s_+}{3}L_4,\quad  L_1+\frac{1 }{2}L_2+\frac{1}{2}L_3+\frac{2s_+}{3}L_4>0,\\&   L_1+ L_2+ \frac 12 L_3 > \frac{s_+}{3}L_4.\nonumber
	\end{align}

 In this paper,
we extend that result in the Oseen-Frank energy  density to the Q-tensor  using the rotational invariant property such that for the condition (\ref {Er1}) on elastic constants $L_1, \cdots , L_4$, we can recover the coercivity condition on the Landau-de Gennes  energy density and establish that:
\begin{theorem}
\label{Theorem 2}
Assume that $L_1$, $L_2$, $L_3$ and $L_4$ satisfy the condition (\ref {Er1}).
 Then for each $Q\in S_*$,    $f_E(Q,\nabla Q)$ is equivalent to a new form
\begin{equation}\label{ND}
  f_{E,2}(Q,\nabla Q):=\frac{ \alpha}{2}|\nabla Q|^2 +V(Q, \nabla Q). \end{equation}
Here $V(Q, \nabla Q)$ is  a sum of square terms as in \eqref{Full W(Q,na Q)} and $\alpha$ is given by
\begin{align*}\alabel{eq alpha}
&\alpha=\min\{2L_1+L_2+L_3-\frac{2s_+}{3}L_4,\, 2L_1-\frac{2s_+}{3}L_4,\,
2L_1+L_2+L_3+\frac{4s_+}{3}L_4\} >0
.\end{align*}

\end{theorem}

\begin{cor}
For the case that $\min\{k_1,k_2,k_3\}\geq k_2+k_4=:\tilde \alpha >0$ (c.f. \cite[p.~551]{HKL}, \cite[p.~467]{GMS}),     we know   that
\[W(u,\na u)=\frac {\tilde \alpha }2 |\na u|^2+ V(u,\na u)\] with
	\begin{align*}
  V(u,\na u)
	=\frac {k_1-\tilde \alpha} 2(\div  u)^2 +\frac {k_2-\tilde \alpha} 2(u\cdot \curl u)^2+\frac {k_3-\tilde \alpha }2 |  u\times   \curl  u|^2.
	\end{align*}
Then the explicit  form of $V(Q,\na Q)$ in \eqref{ND} is
 \begin{align*}\alabel{V-Giaq}
		  V(Q,\na Q) =  &(L_1+\frac{L_2}{2}+\frac{L_3}{2}-\frac{s_+}{3}L_4-\frac 12\alpha)\sum^3_{k=1}\left(\sum^3_{i,j=1}(s^{-1}_+Q_{kj}+\frac13\delta_{kj})\nabla_iQ_{ij}\right)^2
		  		\\
		  		&+  (L_1-\frac{s_+}{3}L_4-\frac 12\alpha)\( \sum^3_{i,j=1} (s^{-1}_+Q_{ij}+\frac13\delta_{ij})(\curl Q_j)_i\)^2
		  		\\
		  		&+(L_1+\frac{L_2}{2}+\frac{L_3}{2}+\frac{2s_+}{3}L_4-\frac 12\alpha)\left|\sum^3_{j=1}(s^{-1}_+Q-\frac 13I)_j\times\curl Q_{j}\right|^2,
		\end{align*}
where  $Q_i$  denotes the $i$-th column of the $Q$ matrix,   $\alpha$ is defined in \eqref{eq alpha} and assume that $L_3\leq 0$.
\end{cor}

\begin{remark}The form $V(Q,\na Q)$  in Corollary 1  is similar to the recent work of Golovaty et al. \cite[p.~8]{GNS20}.  Physical interpretation on   fourth order terms in \eqref{V-Giaq} was numerically analysed in \cite{GKLNS}.\end{remark}

	 By the new form  of $ f_{E,2}(Q,\nabla Q)$ in (\ref{ND}) for each  $Q\in W^{1,2}(\Omega, S_0)$, we can also introduce a new Landau-de Gennes  energy functional
\begin{equation}\label{NLG}	
E_{L,2}(Q; \Omega )=\int_{\Omega}\(  f_{E,2}(Q, \nabla Q) +\frac 1 L f_B(Q)\)\,dx.
\end{equation}
Then we have a similar result in Theorem 1.

It is not clear that each minimizer $Q_L$ of $E_{L}(Q; \Omega )$ or $E_{L,2}(Q; \Omega )$ in $W_{Q_0}^{1,2}(\Omega, S_0)$ is bounded.
Therefore, the energy density $f_{E,1}(Q, \nabla Q)$ in (\ref {LG1}) or $f_{E,2}(Q, \nabla Q)$  in (\ref{ND}) is not bounded above by $C|\nabla Q|^2 +C$. Without this above growth condition on the density,  it is  well-known  that a minimizer $Q_L$ of  the   Landau-de Gennes  energy functional in $W^{1,2}_{Q_0}(\Omega; S_0 )$ does not satisfy the Euler-Lagrange equation in $W^{1,2}(\Omega, S_0)$.
To overcome this difficulty, we introduce a smooth cut-off function $\eta (r)$ in $[0, \infty )$ so that $\eta (r)=1$ for $r\leq M$ with a very large constant $M>0$ and $\eta (r)=0$ for $r\geq M+1$.
Then we modify the Landau-de Gennes  density by
\begin{equation}\label{MD}
\widetilde f_E(Q, \nabla Q):= \frac {\alpha}2 |\nabla Q|^2 +\tilde V(Q, \nabla Q)=  \frac {\alpha}2|\nabla Q|^2 +\eta (|Q|) V(Q, \nabla Q)
 \end{equation}
with the property that
\[ \frac {\alpha}2  |\nabla Q|^2\leq \widetilde f_E(Q, \nabla Q)\leq C|\nabla Q|^2. \]
For a large $M>0$ in (\ref{MD}),  we consider a modified   Landau-de Gennes functional
\begin{equation}\label{MDLG}
\tilde E_{L}(Q; \Omega )=\int_{\Omega}\( \widetilde f_E(Q, \nabla Q) +\frac 1 L f_B(Q)\)\,dx.
\end{equation}
Each minimizer  $Q_L$ of   the modified   Landau-de Gennes energy functional (\ref{MDLG}) in  $W^{1,2}_{Q_0}(\Omega; S_0)$ satisfies the    Euler-Lagrange equation

	\begin{align}\label{MDEL}
		\alpha  \Delta Q_{ij}+\frac 1 2  \na_k (\tilde V_{p^k_{ij}}&+\tilde V_{p^k_{ji}})-\frac 1 3 \delta_{ij}\sum^3_{l=1} \nabla_k \tilde V_{p^k_{ll}}  -\frac 1 2 (\tilde V_{Q_{ij}}  +\tilde V_{Q_{ji}})+   \frac 1 3 \delta_{ij}\sum^3_{l=1} \tilde V_{Q_{ll}}\nonumber
		\\
		=& \frac 1 L \(-aQ_{ij}- b  (Q_{ik}Q_{kj}-\frac 13\delta_{ij}\tr(Q^2)) +cQ_{ij}\tr(Q^2)\)
	\end{align}
in the weak sense.

  \begin{remark} Any weak solution $Q_L$ of (\ref {MDEL})  with boundary vale $Q_0\in W^{1,2}(\Omega ,S_*)$ is uniformly bounded; i.e. for a sufficiently large $M>0$, $|Q_L|\leq M+1$. By using the result of Giaquinta-Giusti \cite{GG} (see also \cites{Gi,Giu}),  it implies that $Q_L$ is partially regular inside $\Omega$. \end{remark}

The Landau-de Gennes theory  is also related to the study of the Ginzburg-Landau approximation.
 The Ginzburg-Landau functional  was introduced in  \cite{GL} to study the phase transition in superconductivity.  For a  parameter $\varepsilon>0$,  the Ginzburg-Landau functional of $u:\Omega\rightarrow \R^3$ is defined by
\begin{equation}\label{GLF}
E_{\varepsilon}(u; \Omega):=\int_{\Omega}\(\frac 12 |\nabla u|^2+\frac{1}{4\varepsilon^2}(1-|u|^2)^2\)\, dx.
\end{equation}
The Euler-Lagrange equation is
\begin{equation}\label{GL}
\Delta u_{\varepsilon}+\frac 1 {\varepsilon^2} u_{\varepsilon}(1-|u_{\varepsilon}|^2)=0.
\end{equation}
 In particular,
using the cross product, the equation (\ref {GL}) becomes
\begin{equation*}
\nabla \cdot (u_{\varepsilon}\times \nabla u_{\varepsilon}) =0.
\end{equation*}
 Chen \cite {Chen} proved that as $\varepsilon\to 0$,  solutions $u_{\varepsilon}$ of the Ginzburg-Landau system (\ref{GL}) weakly converge  to a harmonic map in $W^{1,2}(\Omega;  \R^3)$. Moreover,  Chen and Struwe \cite{CS} proved global existence of partial regular  solutions to the heat flow of harmonic maps using the Ginzburg-Landau approximation.

By comparing with the result of Chen \cite {Chen} (see also \cite{CHH}) on the weak convergence of solutions of the  Ginzburg-Landau equations,   it is very interesting  to study whether the solutions $Q_L$ of the Landau-de Gennes equations (\ref{MDEL}) with a uniform bound of the energy, i.e. $\tilde E_L(Q_L;\Omega)\leq C$ for a uniform constant $C>0$,
		 converge   weakly   to a  solution  $Q_*$  of $\eqref{EL}$ in $W^{1,2}_{Q_0}(\Omega; S_0 )$.   However, it seems  that  the problem is not clear when  $L_2$ and $L_3$  are not zero.
  Under a strong condition, we  solve this problem  to prove:
	\begin{theorem} \label{Theorem 3}
Let $Q_L$ be a weak solution to the equation \eqref{MDEL}.
Assume that the solution $Q_L$
		 converges  strongly to     $Q_*$ in $W^{1,2}_{Q_0}(\Omega; S_0 )$ as $L\to 0$ and satisfies
\begin{align}
	\lim_{L\to 0}\frac 1 L\int_{\Omega}  \tilde f_B(Q_L) \,dx=0.\label{eq Theorem 3}
\end{align}
Then, $Q_*$  is a weak solution to \eqref{EL}.

\end{theorem}

In the proof of Theorem \ref {Theorem 3},   we show that for any $Q\in S_*$, the Hessian of the bulk density $\tilde f_B(Q)$ is positive definite for a uniform constant. As in \cite{CS}, we note that in a neighbourhood   $S_{\delta}$ of the space $S_*$, there is a smooth projection $\pi$. Then we employ Taylor's expansion and  Egoroff's theorem to prove  Theorem \ref {Theorem 3}.

The paper is organized as follows. In Section 2, we prove   Theorem  1. In Section 3, we prove   Theorem  2.  In Section 4, we prove Theorem 3.

\section{Proof of Theorem 1 and the Euler-Lagrange equation}
 \begin{lemma} \label{Lemma 2.1}
 For a uniaxial $Q\in S_*$  of  the form
	 	\[Q= s_+(u\otimes u-\frac 13 I),\quad   u\in S^2,\]
	 	the elastic potential  $f_E(Q,\nabla Q)$ in \eqref{LG} satisfies
 \begin{align*}\alabel{f_E form}
 f_E(Q,\nabla Q)=&\(\frac{L_1}{2}-\frac{s_+L_4}{3}\)\sum_{i,j,k}\(\frac{\p Q_{ij}}{\p x_k}\)^2
		+\frac{L_2}{2}\sum_{i,j,k}\frac{\p Q_{ij}}{\p x_j}\frac{\p Q_{ik}}{\p x_k}\\
&
		+\frac{L_3}{2}\sum_{i,j,k,l}\frac{\p Q_{ik}}{\p x_j}\frac{\p Q_{ij}}{\p x_k}+\frac{3L_4}{2s_+}\sum_{i,j,k,l,n}Q_{ln}Q_{kn}\frac{\p Q_{ij}}{\p x_l}\frac{\p Q_{ij}}{\p x_k}.
\end{align*}
	 \end{lemma}
\begin{proof} Using the fact that $|u|=1$, we have
	\begin{align*}\alabel{ID}
	& Q_{ln}Q_{kn}= s_+^2(u_ku_n-\frac13 \delta_{kn})(u_lu_n-\frac13 \delta_{ln})
	\\
	=&   s_+^2\(u_ku_lu_nu_n-\frac13 \delta_{kn}u_lu_n-\frac13 \delta_{ln}u_ku_n+\frac19 \delta_{kn}\delta_{ln}\)
	\\
	=&s^2\(\frac13u_ku_l+\frac19 \delta_{kl}\) =\frac {s_+}3 s_+(u_ku_l-\frac 13 \delta_{lk})+\frac {2s_+^2}9 \delta_{kl}
	\\
	=&\frac {s_+}3 Q_{kl}+\frac {2s_+^2}9 \delta_{kl}.
	\end{align*}
	Applying the identity \eqref{ID} to the $L_4$ term of (\ref  {LG}), we obtain
	\begin{align*}\alabel{Q third}
	&Q_{lk}\frac{\p Q_{ij}}{\p x_l}\frac{\p Q_{ij}}{\p x_k}=
	(\frac{3}{s_+}Q_{ln}Q_{kn}-\frac {2s_+}3 \delta_{kl})\frac{\p Q_{ij}}{\p x_l}\frac{\p Q_{ij}}{\p x_k}\\
=&\frac{3}{s_+}(Q_{ln}\frac{\p Q_{ij}}{\p x_l})(Q_{kn}\frac{\p Q_{ij}}{\p x_k})-\frac {2s_+}3|\na Q|^2.
	\end{align*}
	Substituting \eqref{Q third} into \eqref{LG}, we prove \eqref{f_E form}.
	\end{proof}

Recall from Longa et al. \cite[p. ~778]{LMT} that we define
	\begin{align*}
	L^{(4)}_5=&Q_{\alpha\rho}Q_{\rho\beta}\frac{\p Q_{\alpha\mu}}{\p x_ \beta}\frac{\p Q_{\mu\nu}}{\p x_\nu},
	\quad
	L^{(4)}_6=Q_{\alpha\rho}Q_{\rho\beta} \frac{\p Q_{\alpha\mu}}{\p  x_ \mu}  \frac{\p Q_{\beta\nu}}{\p x_ \nu },
	\\
	L^{(4)}_7=&Q_{\alpha\rho}Q_{\rho\beta}  \frac{\p Q_{\alpha\mu}}{\p x_ \nu}\frac{\p  Q_{\beta\mu}}{\p x_\nu}.
	\end{align*}
Then we have
 \begin{lemma} \label{Lemma 2.1.b} For a uniaxial $Q\in S_*$, we obtain
\begin{align*}\alabel{Q}
	& Q_{ln}Q_{kn}\frac{\p Q_{ij}}{\p x_l}\frac{\p Q_{ij}}{\p x_k}=\frac{8}{5} L^{(4)}_5- \frac{2}{5}L^{(4)}_6 +\frac{2}{5}L_7^{(4)}.
	\end{align*}
\end{lemma}
\begin{proof} Let $Q= s_+(u\otimes u-\frac13 I)$ for $u\in S^2$.
Noting that  $u_i\na u_i =0$,  we calculate
\begin{align*}
	(u\times \curl u)_1^2=&[u_2(\na_1u_2-\na_2u_1)-u_3(\na_3u_1-\na_1u_3)]^2
	\\
	=&(-u_1\na_1u_1-u_2\na_2u_1-u_3\na_3u_1)^2=[(u\cdot \na)u_1]^2.
\end{align*}
Similarly, we can calculate other terms to obtain
\begin{align*}
	\sum_i[(u\cdot \na)u_i]^2=\sum_i(u\times \curl u)_i^2=|u\times \curl u|^2.
\end{align*}
Moreover, we calculate
	\begin{align*}\alabel{Q 3}
	& Q_{lk}\frac{\p Q_{ij}}{\p x_l}\frac{\p Q_{ij}}{\p x_k}=s_+^{3}  (u_lu_k-\frac 13\delta_{lk})\na_l(u_iu_j)\na_k(u_iu_j)
	\\
	=& s_+^{3} (u_lu_k-\frac 13\delta_{lk})(u_j\na_lu_i+u_i\na_lu_j)(u_j\na_ku_i+u_i\na_ku_j)
	\\
	=& s_+^{3} (u_lu_k-\frac 13\delta_{lk})(\na_lu_i\na_ku_i+\na_lu_j\na_ku_j)
	\\
	=&2s_+^{3} \sum_i[( u\cdot \na) u_i]^2-\frac 23s_+^{3} |\na u|^2=2s_+^{3} |u\times \curl u|^2-\frac 23s_+^{3} |\na u|^2.
	\end{align*}

It follows from using \eqref{Q third} and \eqref{Q 3}   that
	\begin{align}\label{Q4}
 Q_{ln}Q_{kn}\frac{\p Q_{ij}}{\p x_l}\frac{\p Q_{ij}}{\p x_k}&=\frac 1{3}s_+ Q_{lk}\frac{\p Q_{ij}}{\p x_l}\frac{\p Q_{ij}}{\p x_k}+  \frac {2s^2_+}9|\na Q|^2\\
	&=\frac 2{3}s_+^{4}|u\times \curl u|^2+\frac {2}9s_+^4|\na u|^2.\nonumber
	\end{align}
We verify from    \cite[p.~788]{LMT} that
\begin{align*}\alabel{Longa}
4L_5^{(4)}-L_6^{(4)}=\frac {5s_+^4}3|u\times \curl u|^2,\quad L_7^{(4)}=\frac 59s_+^4|\na u|^2.
	\end{align*}

Substituting  \eqref{Longa}  into \eqref{Q4}, we have
	\[Q_{ln}Q_{kn}\frac{\p Q_{ij}}{\p x_l}\frac{\p Q_{ij}}{\p x_k}=\frac{2}{5}(4L^{(4)}_5-L^{(4)}_6)+\frac{2}{5}L_7^{(4)}.\]
	\end{proof}

	Now we give the proof of Theorem 1.
\begin{proof} Under the condition on $L_1, \cdots, L_4$ in Theorem \ref {Theorem 1},    it is clear that \[f_{E,1}(Q,\nabla Q)\geq (\frac{L_1}{2}-\frac{|L_3|}{2}-\frac{s_+L_4}{3})|\na Q|^2, \quad \forall Q\in S_0.\]
By the standard theory of calculus of variations \cite{Ga}, there is a  minimizer $Q_{L}$  of $E_L$ in $W_{Q_0}^{1,2}(\Omega ; S_0)$.
For each    $Q\in W^{1,2}_{Q_0}(\Omega ; S_0)$, we set
\[E(Q ;\Omega ):=\int_{\Omega} f_{E,1}(Q,\nabla Q)\,dx.\]
It implies that
\[E(Q_L ;\Omega )+ \int_{\Omega}(f_B(Q_L)-  \inf_{S_0} f_B)\,dx  \leq  E(Q ;\Omega )
\]
for any  $Q\in W^{1,2}_{Q_0}(\Omega ; S_*)$ with the fact that $\tilde f_B(Q)=f_B(Q)-  \inf_{S_0} f_B=0$.

As $L\to 0$, minimizers  $Q_{L}$  converge  (possible passing
subsequence)  weakly to a tensor $Q_*\in W^{1,2}(\Omega ; S_0)$ with that $f_B( Q_*)=0$, which implies that $Q_*\in S_*$ a.e. in $\Omega$. Then, for any  $Q\in W_{Q_0}^{1,2}(\Omega ; S_*)$, we have
$$E (Q_* ;\Omega )\leq \liminf_{L\to 0}E(Q_{L} ;\Omega ) \leq \limsup_{L\to 0}E(Q_{L} ;\Omega )\leq E(Q;\Omega).
$$
  Therefore $Q_*$ is also a minimizer of $E$ in $W^{1,2}_{Q_0}(\Omega ; S_*)$.
Choosing $Q=Q_*$ in above inequality, it implies that
\[E (Q_* ;\Omega )=\lim_{L\to 0}E_L(Q_{L} ;\Omega ),\quad \lim_{L\to 0}\frac 1 L\int_{\Omega} \tilde f_B(Q_{L}) \,dx=0.\]
 Moreover, it is known that
\begin{align*}
 \int_{\Omega} |\nabla Q_*|^2\,dx\leq& \liminf_{L\to 0}\int_{\Omega} |\nabla Q_L|^2\,dx ,\\
  \int_{\Omega} V(Q_*, \nabla Q_*)\,dx\leq& \liminf_{L\to 0}\int_{\Omega} V(Q_{L},\nabla Q_L )\,dx.
\end{align*}

It implies that $\int_{\Omega} |\nabla Q_*|^2\,dx= \liminf_{L\to 0}\int_{\Omega} |\nabla Q_L|^2\,dx$. Otherwise, there is a subsequence $L_k\to 0$ that
\[\int_{\Omega} |\nabla Q_*|^2\,dx< \lim_{L_k\to 0}\int_{\Omega} |\nabla Q_{L_k}|^2\,dx.\]
Then
\begin{align*}
		&
E (Q_* ;\Omega )=\lim_{L_k\to 0}E_{L_k}(Q_{L_k} ;\Omega ),\\
=&\(\frac{L_1}{2}-\frac{|L_3|}{2}-\frac{s_+L_4}{3}\)\lim_{L_k\to 0}\int_{\Omega} |\nabla Q_{L_k}|^2\,dx+ \lim_{L_k\to 0}\int_{\Omega} V(Q_{L_k}, \nabla Q_{L_k})\,dx\\
< &E (Q_* ;\Omega ).
\end{align*}
 This is impossible. Therefore,  minimizers $Q_{L_k}$ strongly converge, up-to a subsequence, to a minimizer $Q_*=s_+(u_*\otimes u_*-\frac 1 3 I)$ of $E$ in
$W^{1,2}_{Q_0}(\Omega ; S_0)$. Following from the next lemma, $Q_*$ satisfies (\ref{Er}). Applying the result of Dickmann,  $u_*$ is a minimizer of the Oseen-Frank energy in $W^{1,2}(\Omega; S^2)$. Due to the well-known result of Hardt-Kinderlehrer-Lin \cite{HKL}, $u_*$ is partially regular in $\Omega$ (see also \cite{Ho1}). Thus $Q_*$ is partially regular.
\end{proof}

	\begin{lemma}\label{Lemma 2.2} If $Q$ is a minimizer of $E$ in $W^{1,2}_{Q_0}(\Omega; S_*)$, it   satisfies
\begin{align*}
		&
		\quad \bar \alpha \(-s_+ \Delta Q_{ij} +2 \na_kQ_{il}\na_kQ_{jl} -2 s_+^{-1}(Q_{ij}+\frac {s_+}3\delta_{ij})|\na Q|^2\)\\
		&-\nabla_k \(
	(Q_{jl}+\frac {s_+}3 \delta_{jl})V_{p^k_{il}}
	+(Q_{il}+\frac {s_+}3 \delta_{il})V_{p^k_{jl}}
  -2s_+^{-1} (Q_{ij}+\frac {s_+}3 \delta_{ij})(Q_{lm}+\frac {s_+}3 \delta_{lm})V_{p^k_{lm}} \)
		\\
		& + V_{p^k_{il}}\na_kQ_{jl}
	+ V_{p^k_{jl}}\na_k Q_{il}
  -2s_+^{-1}V_{p^k_{lm}}\(\nabla_k Q_{lm} ( Q_{ij}+\frac {s_+}3 \delta_{ij})+( Q_{lm}+\frac {s_+}3 \delta_{lm})\nabla_k Q_{ij}\)
		\\
&+V_{Q_{il}}(Q_{jl}+\frac {s_+}3 \delta_{jl})
	+V_{Q_{jl}}(Q_{il}+\frac {s_+}3 \delta_{il})
  -2s_+^{-1}V_{Q_{lm}}(Q_{lm}+\frac {s_+}3 \delta_{lm}) (Q_{ij}+\frac {s_+}3 \delta_{ij})
\\
		& =0
	\end{align*}
	in the weak sense.
	\end{lemma}
\begin{proof} Let $\phi\in C^\infty_0(\Omega; \R^3)$ be  a test function. For each $u_t = \frac{u+t\phi}{|u+t\phi|}$ with $t\in \R$,
		define
		\begin{align}\label{eq Q variation}
		Q_t(x)=Q(u_t(x))  = s_+\(u_t\otimes u_t-\frac 1 3 I\)\in S_*.
		\end{align}  For any $\eta\in C^\infty_0(\Omega; S_0)$, we choose  a test function $\phi$ such that $\phi_{i}:=u_k\eta_{ik} $. If $Q$ is a minimizer, the    first variation of the energy of $Q$ is zero; i.e.
	\begin{align*}
	\left.\frac{\d }{\d t}\int_\Omega f_E(Q_t,\nabla Q_t)\d x\right|_{t=0}=\int_\Omega\left. f_{Q_{t; i,j}}\frac{\d Q_{t; i,j}}{\d t} +f_{p_t^k}\frac{\d }{\d t}\frac{\p  Q_{t; i,j}}{\p x^k}\d x \right|_{t=0}=0.
	\end{align*}
Note that
\begin{align*}
		   \frac{\d  Q_{t; i,j}}{\d t} =&s_+\(\frac {\phi_i(u_j+t\phi_j)+ (u_i+t\phi_i)\phi_j } {|u+t \phi |^2} -\frac {\big(2(u\cdot \phi)+2t|\phi |^2\big)(u_i +t\phi_i ) (u_j+t\phi_j)}{|u+t\phi |^4}\)
		   \\
		   =&\frac{\Big(( Q_{jl}+\frac {s_+}3 \delta_{jl})+t(Q_{lm}+\frac {s_+}3 \delta_{lm})\eta_{il}+((Q_{il}+\frac {s_+}3 \delta_{il}+t(Q_{lm}+\frac {s_+}3 \delta_{lm})\Big)\eta_{jl}}{2tQ_{il}\eta_{il}+t^2(Q_{lm}+\frac {s_+}3 \delta_{lm})\eta_{il}\eta_{im}}
		   \\
		   &-\frac{2s_+^{-1}((Q_{ij}+\frac {s_+}3 \delta_{ij})+t(Q_{lm}+\frac {s_+}3 \delta_{lm})\eta_{il}\eta_{im})((Q_{lm}+\frac {s_+}3 \delta_{lm})\eta_{lm})}{|2tQ_{il}\eta_{il}+t^2(Q_{lm}+\frac {s_+}3 \delta_{lm})\eta_{il}\eta_{im}|^2}\nonumber
		\end{align*}
where we used  the fact that $|u|=1$ and $\phi_{i}:=u_l\eta_{il} $. We also observe
\begin{align}\label{V1}
&\left.\frac{\d  Q_{t; i,j}}{\d t} \right|_{t=0}=s_+\big( u_j\phi_i+u_i\phi_j  -2(u\cdot \phi)(u_i u_j)\big)
\\
=&( Q_{jl}+\frac {s_+}3 \delta_{jl})\eta_{il}+(Q_{il}+\frac {s_+}3 \delta_{il})\eta_{jl}
-2s_+^{-1}(Q_{ij}+\frac {s_+}3 \delta_{ij})(Q_{lm}+\frac {s_+}3 \delta_{lm})\eta_{lm}.\nonumber
\end{align}
Noting the fact that $\nabla_k |u+t\phi |^2 =0$ at $t=0$  and  substituting $\phi_{i}:=u_l\eta_{il} $, a calculation shows
\begin{align}\label{V2}
		   &  \left.\frac{\d }{\d t}\frac{\p   Q_{t; ij}}{\p x_k}\right|_{t=0}=\left . \( \frac{\p } {\p x_k}\frac{\d }{\d t}    Q_{t; ij}\)\right|_{t=0}
\\
=&\left . s_+\nabla_k\(\frac {\phi_iu_j+ u_i\phi_j+2t\phi_i\phi_j } {|u+t \phi |^2} -\frac {2\(u\cdot \phi +t|\phi |^2\)(u_i +t\phi ) (u_j+t\phi_j)}{|u+t\phi |^4}\) \right |_{t=0}\nonumber
\\
=&s_+\frac{\p }{\p x_k}\(u_j\phi_i+u_i\phi_j -2(u\cdot \phi)u_iu_j\)\nonumber
\\
		    =&\frac{\p }{\p x_k}\((Q_{jl}+\frac {s_+}3 \delta_{jl})\eta_{il}+(Q_{il}+\frac {s_+}3 \delta_{il})\eta_{jl}
  -2s_+^{-1}(Q_{ij}+\frac {s_+}3 \delta_{ij}) (Q_{lm}+\frac {s_+}3 \delta_{lm})\eta_{lm}\)
           \nonumber \\
            =&\frac{\p Q_{jl}}{\p x_k} \eta_{il}+\frac{\p Q_{il}}{\p x_k}\eta_{jl}
  -2s_+^{-1}\(\frac{\p Q_{ij}}{\p x_k}Q_{lm}+\frac{\p Q_{lm}}{\p x_k}(Q_{ij}+\frac {s_+}3 \delta_{ij})\)\eta_{lm}
           \nonumber \\
            &+(Q_{jl}+\frac  {s_+}3 \delta_{jl})\frac{\p \eta_{il}}{\p x_k}+(Q_{il}+\frac {s_+}3 \delta_{il})\frac{\p \eta_{jl}}{\p x_k}
  -2s_+^{-1}( Q_{ij}+\frac {s_+}3 \delta_{ij})(Q_{lm}+\frac {s_+}3 \delta_{lm})\frac{\p \eta_{lm}}{\p x_k}.\nonumber
		    \end{align}

For the special case of the functional $\frac 12 \int_{\Omega} {|\na Q|^2}\d x$, it follows from using (\ref{V2})  and   $|u|^2=1$ that
		\begin{align*}
		& \frac{\d}{\d t}\int_{\Omega}\left.\frac{|\na Q_t|^2}{2}\d x\right|_{t=0}=\int_\Omega  \nabla_k Q_{ij} \left. \frac { d \nabla_k Q_{t; ij}} {dt}  \right|_{t=0}\d x
		\\
		=&s_+^{2}\int_{\Omega}\na_k(u_iu_j)[\na_k(u_ju_l)\eta_{il}+\na_k(u_iu_l)\eta_{jl}]\d x
		\\
		&+s_+^{2}\int_{\Omega}(\na_ku_iu_j+u_i\na_ku_j)(u_ju_l\na_k\eta_{il}+u_iu_l\na_k\eta_{jl})\d x
		\\
		&-\int_{\Omega}2 s_+^{-1}(Q_{lm}+\frac {s_+}3\delta_{lm})|\na Q|^2 \eta_{lm}\d x
		\\
		=&\int_{\Omega}2\na_kQ_{ij} \na_kQ_{jl}\eta_{il}-2(s_+^{-1}Q_{lm}+\frac 13\delta_{lm})|\na Q|^2 \eta_{lm}\d x
		\\
		&+s_+^{2}\int_{\Omega}\na_k u_iu_l\na_k\eta_{il}+\na_ku_ju_l\nabla_k \eta_{jl} \d x
		\\
		=&\int_{\Omega}2 \na_kQ_{il}\na_kQ_{jl}\eta_{ij}-2 (s_+^{-1}Q_{ij}+\frac 13\delta_{ij})|\na Q|^2 \eta_{ij}\d x
		\\
		&+\frac12s_+ \int_{\Omega}(\nabla_k Q_{il} \nabla_k\eta_{il} + \nabla_k Q_{lj}\nabla_k\eta_{jl})\d x\\
		=&\int_{\Omega}\(-s_+ \Delta Q_{ij} +2 \na_kQ_{il}\na_kQ_{jl} -2 (s_+^{-1}Q_{ij}+\frac 13\delta_{ij})|\na Q|^2\) \eta_{ij}\d x\alabel{EL for one constant on S_*}
		\end{align*}
		for all $\eta$ with $\eta_{ij}=\eta_{ji}$.  This means that $Q$ is a weak solution to
		\[ s_+\Delta Q_{ij}-2 \na_kQ_{il}\na_kQ_{jl}+2 (s_+^{-1}Q_{ij}+\frac 13\delta_{ij})|\na Q|^2=0. \]
	For the term $V(Q,\na Q)$, using  (\ref{V1})-(\ref{V2})  and integrating by parts, we have
	\begin{align*}
	&\int_\Omega\left.\frac{\d }{\d t}V(Q_t,\nabla Q_t)\right|_{t=0}\d x=\int_\Omega \left. [V_{p^k_{ij}} \frac { d \nabla_k Q_{ij}} {dt}+V_{Q_{ij}}\frac {d  Q_{ij}}{dt} ]\right|_{t=0}\d x
	\alabel{EL for V on S_*}
	\\
	=&\int_\Omega V_{p^k_{ij}}\((Q_{jl}+\frac {s_+}3 \delta_{jl})\frac{\p \eta_{il}}{\p x_k}+(Q_{il}+\frac {s_+}3 \delta_{il})\frac{\p \eta_{jl}}{\p x_k}+\frac{\p Q_{jl}}{\p x_k} \eta_{il}+\frac{\p Q_{il}}{\p x_k}\eta_{jl}\)\d x
  \\
  		&-2s_+^{-1}\int_\Omega
  	  V_{p^k_{ij}}\(\frac{\p Q_{ij}}{\p x_k} (Q_{lm}+\frac {s_+}3 \delta_{lm})+\frac{\p Q_{lm}}{\p x_k}(Q_{ij}+\frac {s_+}3 \delta_{ij})\)\eta_{lm}\d x
  	  \\
  &+\int_\Omega -2s_+^{-1}V_{p^k_{ij}}(Q_{ij}+\frac {s_+}3 \delta_{ij}) (Q_{lm}+\frac {s_+}3 \delta_{lm})\frac{\p \eta_{lm}}{\p x_k}+V_{Q_{ij}} ( Q_{jl}+\frac {s_+}3 \delta_{jl})\eta_{il}\d x
	\\
	&+\int_\Omega V_{Q_{ij}}\((Q_{il}+\frac {s_+}3 \delta_{il})\eta_{jl}
  -2s_+^{-1}(Q_{ij}+\frac {s_+}3 \delta_{ij}) (Q_{lm}+\frac {s_+}3 \delta_{lm})\eta_{lm} \) \d x
	\\
	=&-\int_\Omega \frac{\p }{\p x_k}\(
	(Q_{jl}+\frac {s_+}3 \delta_{jl})V_{p^k_{il}}
	+(Q_{il}+\frac {s_+}3 \delta_{il})V_{p^k_{jl}}\)\eta_{ij}
  \d x
  \\
  &+\int_\Omega \frac{\p }{\p x_k}\(2s_+^{-1} (Q_{ij}+\frac {s_+}3 \delta_{ij})(Q_{lm}+\frac {s_+}3 \delta_{lm})V_{p^k_{lm}} \)\eta_{ij}+V_{p^k_{il}}\frac{\p Q_{jl}}{\p x_k}\eta_{ij}\d x
	\\
	&+\int_\Omega\(
	  	V_{p^k_{jl}}\frac{\p Q_{il}}{\p x_k}
  -2s_+^{-1}V_{p^k_{lm}}\(\frac{\p Q_{lm}}{\p x_k}(Q_{ij}+\frac {s_+}3 \delta_{ij})+( Q_{lm}+\frac {s_+}3 \delta_{lm})\frac{\p Q_{ij}}{\p x_k}\)\)\eta_{ij}\d x
  \\
	&+\int_\Omega \( V_{Q_{il}}(Q_{jl}+\frac {s_+}3 \delta_{jl})
	+V_{Q_{jl}}(Q_{il}+\frac {s_+}3 \delta_{il})\)\eta_{ij}\d x
  \\
  &-2s_+^{-1}\int_\Omega
    V_{Q_{lm}}(Q_{lm}+\frac {s_+}3 \delta_{lm}) (Q_{ij}+\frac {s_+}3 \delta_{ij})\eta_{ij}\d x.
	\end{align*}
	 Combining  above two identities \eqref{EL for one constant on S_*}-\eqref{EL for V on S_*}, we prove  Lemma  \ref{Lemma 2.2}.
\end{proof}

\begin{lemma} Assume that $Q=s_+(u\otimes u-\frac 1 3 I)$. Then $Q=(Q_{ij})$ is a solution of equation
\begin{equation}\label{QH}
\Delta Q_{ij}-2s_+^{-1}\na_kQ_{il}\na_kQ_{jl}+2s_+^{-1}(s_+^{-1}Q_{ij}+\frac 13\delta_{ij})|\na Q|^2=0
\end{equation}
 if and only if $u$ is a harmonic map from $\Omega$ into $S^2$; i.e. $-\Delta u=|\nabla u|^2 u$.
	 \end{lemma}
\begin{proof} Let  $u$ be a harmonic map  from $\Omega$ into $S^2$.
Then we calculate
	\begin{align}
	\Delta Q_{ij}=& s_+\na_k(u_j\na_ku_i+u_i\na_ku_j)\nonumber
	\\
	=& s_+(u_i\Delta u_j+2\na_k u_j\na_ku_i+u_j\Delta u_i)\label{Ha1}
	\\
	=& 2s_+(-|\na u|^2u_iu_j+\na_k u_j\na_ku_i)\nonumber.
	\end{align}
	 Noting that $|\na u|^2=\frac{s_+^{-2}}{2}|\na Q|^2$ and $|u|=1$, we obtain
	\begin{align}\label{Ha2}
	 &\na_k u_j\na_ku_i=\na_k u_j\na_ku_iu_lu_l
	\\
	=&[\na_k(u_ju_l)-u_j\na_k u_l][u_l\na_k u_i] =\na_k(u_ju_l)u_l\na_k u_i
	\nonumber\\
	=&\na_k(u_ju_l)[\na_k(u_lu_i)-u_i\na_ku_l]
	\nonumber\\
	=&s_+^{-2}\na_k Q_{jl}\na_k Q_{il}-(u_j\na_ku_l+u_l\na_ku_j)u_i\na_ku_l
	\nonumber\\
	=&s_+^{-2}\na_k Q_{jl}\na_k Q_{il}-(s_+^{-1}Q_{ij}+\frac 13\delta_{ij})\frac{s_+^{-2}}{2}|\na Q|^2.\nonumber
	\end{align}
	 Substituting (\ref{Ha2}) into (\ref{Ha1}) with the fact that  $|\na u|^2=\frac{s_+^{-2}}{2}|\na Q|^2$, we obtain
	\begin{align*}
	&\Delta Q_{ij}=-2s_+^{-1}(s_+^{-1}Q_{ij}+\frac 13\delta_{ij})|\na Q|^2+2s_+^{-1}\na_k Q_{jl}\na_k Q_{il}.
	\end{align*}
	
Conversely, let $Q$ be a solution to (\ref{QH}).  Using (\ref{Ha1}), (\ref{Ha2}) with the fact that $u_j\Delta u_j=-|\na u|^2$, we have
\begin{align*}
	\Delta u_i=(\Delta (u_i u_j)-u_i\Delta u_j-2\na_k u_j\na_ku_i )u_j =u_i|\nabla u|^2.
	\end{align*}

\end{proof}
\section{The  coercivity and Proof of Theorem 2}

\begin{lemma} \label{Lemma 3.1}
	Assume the Frank constants $k_1,\cdots , k_4$ satisfy the strong Ericksen condition (\ref {Er}); i.e.
	\[k_1>0,\quad k_2>|k_4|,\quad k_3>0,\quad 2k_1>k_2+k_4.\]
	Then for each $u\in S^2$, the density  $W(u, \nabla u)$ of  the form (\ref{OF density}) is equivalent to the new form
	\[W(u, \nabla u)=\frac{\tilde \alpha}{2} |\nabla u|^2+V(u,\nabla u), \]
	where
$V(u,\nabla u)$ is a sum of square terms (see (\ref{EV-form})) satisfying
$$V(u,\nabla u)\leq  C(1+|u|^2)|\na u|^2,\quad |V_u(u,\nabla u)|\leq C(1+|u|)|\nabla u|^2$$ for all $u\in \R^3$  and
\begin{equation}\label{eq tilde alpha}
		\tilde\alpha=\min\left\{ k_2+k_4 ,2k_1-k_2-k_4,k_2-|k_4|,k_3\right\}>0.
\end{equation}\end{lemma}
\begin{proof}
	Note that $W(u, \nabla u)$ is rotational invariant (e.g.  \cite{GMS}); i.e. for each  $R\in SO(3)$, $\tilde x=R(x-x_0)$ and $\tilde u = R  u(x)=R u$. Then we have
	\[
	W(\tilde u, \tilde \na \tilde  u)= W(Ru, R\na  uR^T)=W(u,\na u) .
	\]
Then for any $u\in S^2$, we can find some $R=R(u(x_0))\in SO(3)$  at each point $x_0\in \Omega$ such that
	\[\tilde u(0):=R u(x_0) = (0,0,1)^T.\]

In fact, we can find  the exact form of $R$ at $x_0$ by rotating $\tilde u$ back to $u$ around $x$ and $y$ axes in a $(x,y,z)$ Cartesian coordinates.
		\[R_x:=\begin{pmatrix}1&0&0\\0&\cos\phi&-\sin\phi\\0&\sin\phi&\cos\phi\end{pmatrix}
		\quad R_y:=\begin{pmatrix}\cos\varphi&0&\sin\varphi\\0&1&0\\-\sin\varphi&0&\cos\varphi\end{pmatrix}.\]
Here $\phi\in [-\pi,\pi)$ and  $\varphi\in [-\pi/2,\pi/2)$.
Let $R_1:=(R_xR_y)^T,R_2:=(R_yR_x)^T$. We choose an open cover $\{U_i\}_{i=1}^6$ for the sphere $S^2$ with  open sets
\begin{align*}
U_1=&\{u\in S^2|u_3> \frac 12\},\quad U_2=\{u\in S^2|u_3< -\frac 12\},\alabel{eq partition}
\\
U_3=&\{u\in S^2|u_2> \frac 12\},\quad U_4=\{u\in S^2|u_2< -\frac 12\},
\\
U_5=&\{u\in S^2|u_1> \frac 12\},\quad U_6=\{u\in S^2|u_1< -\frac 12\}.
\end{align*}
Then there is a partition of unity subordinate to the open cover $\{U_i\}_{i=1}^6$; i.e. there exist
$\{\xi_i(u)\}_{i=1}^6$ with  $0\leq \xi_i\leq 1$ having support of $\xi$ in $U_i$ for each $i=1,\cdots 6$.
In particular, $\xi_1(u) =1 $ in $S^2\backslash(\cup_{i=2}^6 U_i)$, $\xi_1(u)\in [0,1]$ in $U_1\cap (\cup_{i=2}^6 U_i)$ and 0 otherwise. Then the rotational invariant energy density can be written as
		\begin{align*}\alabel{partition for W}
			W(\tilde u, \tilde \na \tilde  u)=&\sum_{i=1}^4\xi_i(u) W(R_1u, R_1\na  uR_1^T) +\sum_{j=5}^6\xi_j(u)W(R_2u, R_2\na  uR_2^T).
		\end{align*}
Without loss of generality, we compute $W(u,\na u)$ for the case where $\xi_1(u)=1$. The rotation is
	\begin{align*}
	R^T_1=R_xR_y=&\begin{pmatrix}\cos\varphi&0&\sin\varphi
	\\
	\sin\phi\sin\varphi&\cos\phi&-\sin\phi\cos\varphi
	\\
	-\cos\phi\sin\varphi&\sin\phi&\cos\phi\cos\varphi
	\end{pmatrix}.\alabel{Rotation}
	\end{align*}
	 Then
	\begin{align*}
	u_1(x_0)=\sin\varphi,\quad  u_2(x_0)=-\sin\phi\cos\varphi,\quad u_3(x_0)=\cos\phi\cos\varphi=\cos\phi\sqrt{1-u_1^2(x_0)}.
	\end{align*}
	Then
	\begin{align*}
	\sin\varphi=&u_1(x_0),&&\cos\varphi=\sqrt{u_2^2(x_0)+u_3^2(x_0)},
	\\
	\sin\phi=&\frac{-u_2(x_0)}{\sqrt{u_2^2(x_0)+u_3^2(x_0)}},&&\cos\phi=\frac{u_3(x_0)}{\sqrt{u_2^2(x_0)+u_3^2(x_0)}}.
	\end{align*}
	Therefore, at $x_0$
	\begin{align*}
	R_1(u)=\begin{pmatrix} \sqrt{u_2^2+u_3^2}&\frac{-u_1u_2}{\sqrt{u_2^2+u_3^2}} &\frac{-u_1u_3}{\sqrt{u_2^2+u_3^2}}
		\\
		0 &\frac{u_3}{\sqrt{u_2^2+u_3^2}}& \frac{-u_2}{\sqrt{u_2^2+u_3^2}}
		\\
		u_1& u_2& u_3
		\end{pmatrix}\alabel{Rotation u}.
	\end{align*}
Noting that $|\tilde u|^2=1$ and $\tilde u=(0,0,1)$ at $0$, we have at $0$
	\[\frac{\p \tilde u_3}{\p \tilde x_i} = -(\tilde u_1 \frac{\p \tilde u_1}{\p \tilde x_i}+  \tilde u_2 \frac{\p \tilde u_2}{\p \tilde x_i})=0 \] for all $i=1,2,3$. Then we have at $0$
	\begin{align*}
	|\tilde \na \tilde u|^2 =& |\tilde \na \tilde u_1|^2+|\tilde \na \tilde u_2|^2,
	\quad
	\tilde \na \cdot \tilde u =  \tilde \na_1\tilde u_1+\tilde \na_2\tilde u_2,
	\\
	 \curl\tilde u =&(-\tilde \na_3\tilde u_2,\tilde \na_3\tilde u_1,\tilde \na_1\tilde u_2-\tilde \na_2\tilde u_1),
	\\
	\tr(\tilde \na \tilde u)^2 =&|\tilde \na_1\tilde u_1|^2+|\tilde \na_2\tilde u_2|^2+2\tilde \na_1\tilde u_2\tilde \na_2\tilde u_1.
	\end{align*}
	We evaluate four terms of the Oseen-Frank potential at $0$
	\begin{align*}
	(\tilde \na \cdot \tilde u)^2=&(\tilde \na_1 \tilde u_1+\tilde \na_2 \tilde u_2)^2,
	\\
	(\tilde u\cdot \curl \tilde u)^2=&(-\tilde u_1\tilde \na_3\tilde u_2+\tilde u_2\tilde \na_3\tilde u_1+\tilde u_3(\tilde \na_1\tilde u_2-\tilde \na_2\tilde u_1))^2
	\\
	=& (\tilde \na_1\tilde u_2-\tilde \na_2\tilde u_1)^2 ,
	\\
	|\tilde u\times  \curl \tilde u|^2
	=&\(\tilde u_2(\tilde \na_1\tilde u_2-\tilde \na_2\tilde u_1)-\tilde u_3\tilde \na_3 \tilde u_1\)^2+\(-\tilde u_3\tilde \na_3\tilde u_2-\tilde u_1(\tilde \na_1\tilde u_2-\tilde \na_2\tilde u_1)\)^2
	\\
	&+\(\tilde u_1\tilde \na_3\tilde u_1+\tilde u_2\tilde \na_3\tilde u_2\)^2 =|\tilde \na_3\tilde u_1|^2+|\tilde \na_3\tilde u_2|^2,
	\\
	(\tr(\tilde \na \tilde u)^2-(\tilde \na \cdot \tilde u)^2)=&|\tilde \na_1\tilde u_1|^2+|\tilde \na_2\tilde u_2|^2+2\tilde \na_1\tilde u_2\tilde \na_2\tilde u_1-(\tilde \na_1 \tilde u_1+\tilde \na_2 \tilde u_2)^2
	\\
	=&2\tilde \na_1\tilde u_2\tilde \na_2\tilde u_1-2\tilde \na_1 \tilde u_1\tilde \na_2 \tilde u_2
	\end{align*}
	Substituting above identities into the density, we have
	\begin{align*}
	&\alabel{eq W}2W(\tilde u,\tilde \na \tilde u)=k_1(\div \tilde u)^2+k_2(\tilde u\cdot \curl \tilde u)^2+k_3|\tilde u\times \curl \tilde u|^2
\\&\qquad\,\qquad  +(k_2+k_4)(\tr(\na \tilde u)^2-(\div \tilde u)^2),\\
=&k_1(\tilde \na_1 \tilde u_1+\tilde \na_2 \tilde u_2)^2 +k_2(|\tilde \na_1\tilde u_2|^2+|\tilde \na_2\tilde u_1|^2)
	 +k_3(|\tilde \na_3\tilde u_1|^2+|\tilde \na_3\tilde u_2|^2)\\
	&+2k_4 \tilde \na_1\tilde u_2\tilde \na_2\tilde u_1-2(k_2+k_4)(\tilde \na_1 \tilde u_1\tilde \na_2 \tilde u_2)
	 \\
		=&\frac{2k_1-k_2-k_4}{2}(\tilde \na_1 \tilde u_1+\tilde \na_2 \tilde u_2)^2 +\frac{k_2+k_4}{2}(\tilde \na_1 \tilde u_1-\tilde \na_2 \tilde u_2)^2
		\\
		&+(k_2-|k_4|)(|\tilde \na_1\tilde u_2|^2+|\tilde \na_2\tilde u_1|^2) +|k_4|(\tilde \na_1\tilde u_2+\sign{k_4}\tilde \na_2\tilde u_1)^2
		\\
		&+k_3(|\tilde \na_3\tilde u_1|^2+|\tilde \na_3\tilde u_2|^2)
	\\
	=&\tilde\alpha|\tilde \na \tilde u|^2+\frac{2k_1-k_2-k_4-\tilde\alpha}{2}(\tilde \na_1 \tilde u_1+\tilde \na_2 \tilde u_2)^2 +\frac{k_2+k_4-\tilde\alpha}{2}(\tilde \na_1 \tilde u_1-\tilde \na_2 \tilde u_2)^2
	\\
	&+(k_2-|k_4|-\tilde\alpha)(|\tilde \na_1\tilde u_2|^2+|\tilde \na_2\tilde u_1|^2)  +(k_3-\tilde\alpha)(|\tilde \na_3\tilde u_1|^2+|\tilde \na_3\tilde u_2|^2)
	\\
	&+|k_4|(\tilde \na_1\tilde u_2+\sign{k_4}\tilde \na_2\tilde u_1)^2
	=\tilde\alpha|\tilde \na \tilde u|^2+2V(\tilde u,\tilde \na \tilde u),
	\end{align*}
where $\tilde \alpha$, which is defined in \eqref{eq tilde alpha}, is a positive constant due to the strong Ericksen condition (\ref {Er}). The term $V(u,\na u)$ can be written as
\begin{align}\label{V-form}
2V(\tilde u,\tilde \na \tilde u)&:=\frac{2k_1-k_2-k_4-\tilde\alpha}{2}( \div \tilde u)^2 +(k_3-\tilde\alpha) | \tilde u\times  \curl \tilde u|^2
	\\
	& +\frac{k_2+k_4-\tilde\alpha}{2}(\tilde \na_1 \tilde u_1-\tilde \na_2 \tilde u_2)^2+(k_2-|k_4|-\tilde\alpha)(|\tilde \na_1\tilde u_2|^2+|\tilde \na_2\tilde u_1|^2)\nonumber
	\\
	&+|k_4|(\tilde \na_1\tilde u_2+\sign{k_4}\tilde \na_2\tilde u_1)^2.\nonumber
\end{align}
Using \eqref{Rotation u} for the case of $\xi_1(u)=1$, we find
	\begin{align*}
		&\quad \tilde \na \tilde u = R\na uR^T\\&= \begin{pmatrix} \sqrt{u_2^2+u_3^2}&\frac{-u_1u_2}{\sqrt{u_2^2+u_3^2}} &\frac{-u_1u_3}{\sqrt{u_2^2+u_3^2}}
		\\
		0 &\frac{u_3}{\sqrt{u_2^2+u_3^2}}& \frac{-u_2}{\sqrt{u_2^2+u_3^2}}
		\\
		u_1& u_2& u_3
		\end{pmatrix}\na u \begin{pmatrix} \sqrt{u_2^2+u_3^2}& 0 &u_1
		\\
		 \frac{-u_1u_2}{\sqrt{u_2^2+u_3^2}}  &\frac{u_3}{\sqrt{u_2^2+u_3^2}}& u_2
		\\
		\frac{-u_1u_3}{\sqrt{u_2^2+u_3^2}} &\frac{-u_2}{\sqrt{u_2^2+u_3^2}}  & u_3
		\end{pmatrix}.
		\end{align*}
A direct calculation yields
\begin{align*}
(R\na u)_{1,1}=& \sqrt{u_2^2+u_3^2}\na_1 u_1-\frac{u_1(u_2\na_1 u_2+u_3\na_1u_3)}{\sqrt{u_2^2+u_3^2}} =\frac{\na_1u_1}{\sqrt{u_2^2+u_3^2}},
\\(R\na u)_{1,2}=& \sqrt{u_2^2+u_3^2}\na_2 u_1-\frac{u_1(u_2\na_2 u_2+u_3\na_2u_3)}{\sqrt{u_2^2+u_3^2}} =\frac{\na_2u_1}{\sqrt{u_2^2+u_3^2}},
\\(R\na u)_{1,3}=& \sqrt{u_2^2+u_3^2}\na_3 u_1-\frac{u_1(u_2\na_3 u_2+u_3\na_3u_3)}{\sqrt{u_2^2+u_3^2}} =\frac{\na_3u_1}{\sqrt{u_2^2+u_3^2}},
\\
(R\na u)_{2,1}=& \frac{u_3\na_1u_2-u_2\na_1u_3}{\sqrt{u_2^2+u_3^2}},\quad (R\na u)_{2,2}=\frac{u_3\na_2u_2-u_2\na_2u_3}{\sqrt{u_2^2+u_3^2}},
\\
(R\na u)_{2,3}=&\frac{u_3\na_3u_2-u_2\na_3u_3}{\sqrt{u_2^2+u_3^2}}.
\end{align*}
Note that $u_1^2 \leq \frac  {3|u|^2}4 $ for the case of $\xi_1(u)=1$.  Then it yields
\begin{align*}
\tilde\na_1 \tilde u_1
=&\na_1u_1-\frac{u_1u_2\nabla_2u_1+u_1u_3\nabla_3 u_1} {|u|^2- u_1^2}|u|,
\\
\tilde\na_2 \tilde u_2=&\frac{u_3^2}{ u_2^2+u_3^2}\na_2u_2-\frac{u_3u_2}{ u_2^2+u_3^2}(\na_2u_3+\na_3u_2)+\frac{u_2^2}{ u_2^2+u_3^2}\na_3u_3\\
=&\na_2u_2+\na_3u_3+ \frac{u_1u_2\nabla_2u_1+u_1u_3\nabla_3u_1 }{|u|^2- u_1^2 }|u|,\\
\tilde\na_1 \tilde u_2=&u_3\na_2u_1-u_2\na_3u_1+ \frac{u_1u_2u_3}{ u_2^2+u_3^2}(\na_3u_3-\na_2 u_2)-\frac{u_1u_3^2}{ u_2^2+u_3^2}\na_2u_3
+\frac{u_1u_2^2}{ u_2^2+u_3^2}\na_3 u_2\\
=&(1+ \frac {u_1^2}{|u|^2-u_1^2})  (u_3\na_2u_1-u_2\na_3u_1)=\frac {|u|^2}{|u|^2-u_1^2}  (u_3\na_2u_1-u_2\na_3u_1),
\\
\tilde\na_2 \tilde u_1 =&u_3\na_1u_2-u_2\na_1u_3+\frac{u_1u_2u_3}{ u_2^2+u_3^2}(\na_3u_3-\na_2 u_2)+\frac{u_1u_2^2 }{ u_2^2+u_3^2}\na_2u_3-\frac{u_1u_3^2}{ u_2^2+u_3^2}\na_3u_2\\
=& (u_3\na_1u_2-u_2\na_1u_3)+u_1(\na_2u_3- \na_3u_2)+\frac {u_1^2(u_3\nabla_2u_1-u_2\nabla_3 u_1) }{|u|^2-u^2_1}.
\end{align*}
Substituting the above identities into (\ref{V-form}), for the case of $\xi_1(u)=1$, we see that
\begin{align*}\alabel{EV-form}
&2V(u,\na u)=2V(\tilde u,\tilde \na \tilde u)=\frac{2k_1-k_2-k_4-\tilde\alpha}{2}(  \div  u)^2 +(k_3-\tilde\alpha) |  u\times   \curl  u|^2
\\
& +\frac{k_2+k_4-\tilde\alpha}{2}\(\na_1u_1-\na_2u_2-\na_3u_3-\frac{2u_1|u|(u_2\nabla_2u_1 + u_3\nabla_3 u_1)} {|u|^2- u_1^2}\)^2
\\
&+(k_2-|k_4|-\tilde\alpha)\(\frac {|u|^2}{|u|^2-u_1^2}  (u_3\na_2u_1-u_2\na_3u_1)\)^2
\\
&+(k_2-|k_4|-\tilde\alpha)\( (u_3\na_1u_2-u_2\na_1u_3)+u_1(\na_2u_3- \na_3u_2)+\frac {u_1^2(u_3\nabla_2u_1-u_2\nabla_3 u_1) }{|u|^2-u^2_1}\)^2
\\
&+|k_4|\Bigg(\sign{k_4} \Big((u_3\na_1u_2-u_2\na_1u_3)+u_1(\na_2u_3- \na_3u_2)\Big)
\\
&\qquad \qquad+\frac {|u|^2(1+\sign{k_4} )}{|u|^2-u_1^2} (u_3\na_2u_1-u_2\na_3u_1) \Bigg)^2.
\end{align*}
Note that  \eqref{EV-form} is the form of $V(u,\na u)$ for $\xi_1=1$. One can repeat the calculation for the second rotation $R_2$ in \eqref{partition for W}. To extend \eqref{EV-form} to $u\in \R^3$, we define $\xi_i$ for $\frac{u}{|u|}$ similarly to \eqref{eq partition}. Thus we prove the required   result. Then we find that $V(u,  \na u)$  is quadratic in $\na u$ and $0\leq V(  u,  \na u)\leq C(1+|u|^2)|\na u|^2$ for all $u\in \R^3$, which implies that for all $u\in \R^3$, we have
	\begin{align*}
	&W( u, \na  u)=\frac{\tilde\alpha}2| \na  u|^2+V( u,  \na  u )\geq \frac{\tilde\alpha}2| \na  u|^2,\\
&|V_u(u,\nabla u)|\leq C(1+|u|)|\nabla u|^2.
	\end{align*}
	
\end{proof}
\begin{remark}If the Frank constants satisfy that $min\{k_1,k_2,k_3\}\geq k_2+k_4= \tilde \alpha >0$ and $k_4<0$ as in  \cite[p.~551]{HKL} (see also \cite[p.~467]{GMS}). Then the equation \eqref{EV-form} becomes
	\begin{align*}\alabel{eq Giaq}
  2V(u,\na u)
	=(k_1-\tilde \alpha)(\div  u)^2 +(k_2-\tilde \alpha)(u\cdot \curl u)^2+(k_3-\tilde \alpha)|  u\times   \curl  u|^2.
	\end{align*}
Thus the form $W(u,\na u)$ with the form \eqref{eq Giaq} includes the cases in \cite[p.~551]{HKL} and  \cite[p.~467]{GMS}.
\end{remark}

Next, we will prove Theorem \ref {Theorem 2} by using Lemma \ref{Lemma 3.1}.

Using the form $Q= s_+(u\otimes u-\frac13 I)$ for $u\in S^2$, it can be seen that
	\begin{align*}
	&s_+^{-2}\sum_{i, j,k=1}^3\(\frac{\p Q_{ij}}{\p x_k}\)^2= \sum_{i,j,k=1}^3(u_j\na_ku_i+u_i\na_ku_j)^2 =2|\na u|^2,
	\\
	&s_+^{-2}\sum_{j,k=1}^3\frac{\p Q_{ij}}{\p x_j}\frac{\p Q_{ik}}{\p x_k}=\sum_{j,k=1}^3 (u_j\na_j u_i+u_i\na_ju_j)(u_k\na_ku_i+u_i\na_k u_k)
	\\
	=&(\na\cdot u)^2+\sum_i[( u\cdot \na) u_i]^2=(\na\cdot u)^2+|u\times \curl u|^2,
	\\
	&s_+^{-2}\frac{\p Q_{ik}}{\p x_j}\frac{\p Q_{ij}}{\p x_k}= (u_k\na_ju_i+u_i\na_ju_k)(u_j\na_ku_i+u_i\na_ku_j)
	\\
	=&\tr (\na u)^2+\sum_i[( u\cdot \na) u_i]^2=\tr (\na u)^2+|u\times \curl u|^2,
	\\
	&s_+^{-3} Q_{lk}\frac{\p Q_{ij}}{\p x_l}\frac{\p Q_{ij}}{\p x_k}=2|u\times \curl u|^2-\frac 23|\na u|^2.
	\end{align*}
	Here the last equality is from \eqref{Q 3}.
	Substituting above identities into the form $f_E(Q,\nabla Q)$, we have
	\begin{align}\label{L1}
	f_E(Q,\nabla Q) =&s_+^2L_1|\na u|^2
	+\frac{s_+^2L_2}{2}((\na\cdot u)^2+|u\times \curl u|^2)
	\\
	&+\frac{s_+^2L_3}{2}(\tr (\na u)^2+|u\times \curl u|^2)\nonumber\\
&+s_+{^3}L_4(|u\times \curl u|^2-\frac 13|\na u|^2)\nonumber
	\nonumber\\
	=&(s_+^2L_1-\frac{s_+^3}{3}L_4)|\na u|^2+ \frac{s_+^2}{2}L_2(\na \cdot u)^2\nonumber
	\\
	&+(\frac{s_+^2}{2}L_2+\frac{s_+^2}{2}L_3+s_+{^3}L_4)|u\times \curl u|^2+\frac{s_+^2}{2}L_3\tr(\na u)^2.\nonumber
	\end{align}
For each $u\in S^2$,  note that
\begin{align*}
	|\na u|^2= \tr(\na u)^2+|\curl u|^2,\quad
 |\curl u|^2=(u\cdot \curl u)^2+|u\times \curl u|^2.
\end{align*}
Using the above identities, we have
	\begin{align}\label{W1}
	2W(u,\na u)  &=k_1(\na \cdot u)^2+k_2(u\cdot \curl u)^2+k_3|u\times \curl u|^2\\
&\quad +(k_2+k_4)(\tr(\na u)^2-(\na \cdot u)^2)\nonumber
\\&=k_2|\na u|^2+(k_1-k_2-k_4)(\na \cdot u)^2\nonumber\\
&\quad +(k_3-k_2) |u\times \curl u|^2 +k_4 \tr(\na u)^2.\nonumber
	\end{align}
	
	Similarly to \cite{MGKB}, comparing (\ref{L1}) with (\ref{W1}), we find that for each $Q\in S_*$,  $f_E(Q,\nabla Q)=W(u,\nabla u) $ is true when
	 \begin{align}
\begin{cases}\label{Relation}
	k_1&= 2s_+^{2}L_1+ s_+^2 L_2+s_+^2L_3-\frac{2s_+^{3}}{3}L_4\\
	k_2&=2s_+^2L_1-\frac{2s_+^{3}}{3}L_4\\
	k_3&=2s_+^2L_1+ s_+^2 L_2+ s_+^2 L_3+\frac{4s_+^{3}}{3}L_4\\
	k_4&= s_+^2 L_3
	\end{cases}\Leftrightarrow
\begin{cases}
	L_1&=-\frac 16s_+^{-2}k_1+\frac 12s_+^{-2}k_2+\frac 16s_+^{-2}k_3\\
	L_2&=s_+^{-2}k_1-s_+^{-2}k_2-s_+^{-2}k_4\\
	L_3&=s_+^{-2}k_4\\
	L_4&=-\frac12 s_+^{-3}k_1+\frac 12s_+^{-3}k_3.
	\end{cases}
\end{align}
	Using Lemma \ref {Lemma 3.1}, the density $W(u,\nabla u)$ has a lower bound if the coefficients $k_1, \cdots, k_4$ satisfy the strong Ericksen condition (\ref {Er}). Using the relation (\ref{Relation}) between $k_i$ and $L_i$ with $i=1, ..., 4$ , the strong Ericksen condition (\ref {Er}) is equivalent to that
\begin{align*}
	 & L_1-\frac{1}{2}|L_3|> \frac{s^+}{3}L_4,\quad  L_1+\frac{1 }{2}L_2+\frac{1}{2}L_3+\frac{2s_+}{3}L_4>0,\\&   L_1+ L_2+ \frac 12 L_3 > \frac{s_+}{3}L_4.
	\end{align*}
Now we prove Theorem \ref {Theorem 2}.
\begin{proof}
	 For any $Q(u)= s_+(u\otimes u-\frac13 I)$ with $u\in S^2$,  note that
\[Q(-u)= s_+(-u\otimes -u-\frac13 I)=Q(u),\quad f_E(Q,\nabla Q)=W(u,\nabla u). \]
Therefore, we can assume that $u=(u_1, u_2, u_3)$ with $u_1\geq 0$. For a $Q\in S_*$, there is a unique $u\in S^2$ such that
\begin{align}\label {3.12}
&u_1=\sqrt{|s_+^{-1}Q_{11}+\frac 13   |}, \;  u_2=\sign{Q_{12}} \sqrt{|s_+^{-1}Q_{22}+\frac 13   |},\\
 &u_3=\sign{Q_{13}} \sqrt{|s_+^{-1}Q_{33}+\frac 13   |}.\nonumber
 \end{align}
Using the fact that $|u|^2=1$, a direct calculation yields
\[\nabla_ku_i =u_j\nabla_k (u_i u_j)=s_+^{-1}(u_1\nabla_kQ_{i1}+u_2\nabla_kQ_{i2}+u_3\nabla_kQ_{i3})=s_+^{-1}u_j\nabla_kQ_{ij},
 \]
 which implies
 \begin{align}\label {3.14}
u_l\nabla_ku_i &=\sum_js_+^{-1}(s_+^{-1}Q_{lj}+\frac13\delta_{lj})\nabla_kQ_{ij}=\sum_js_+^{-2}(Q_{lj}+\sqrt{\frac16}|Q|\delta_{lj})\nabla_kQ_{ij}.
 \end{align}
 Here we used the fact that $|Q|=\sqrt{\frac 2 3}s_+$.
Let $S:=\{Q\in S_0 : |Q|=1\}$ be the unit sphere of $S_0$.
 By Cauchy's inequality, we have
\[|Q_{11}|+|Q_{22}|+|Q_{33}|\leq \sqrt 3 (|Q_{11}|^2+|Q_{22}|^2+|Q_{33}|^2)^{1/2}\leq \sqrt 3 |Q|.\]
Consider
\begin{align*}\alabel{PoU Q}
&U_1=\left \{Q\in S: \,|Q_{11}|<\frac  {\sqrt 3} {3}|Q|\right\},\quad U_2=\left\{Q\in S:  \,|Q_{22}|<  \frac {\sqrt 3} {3}|Q|\right\},\\
&U_3=\left\{Q\in S:\, |Q_{33}|<  \frac  {\sqrt 3} {3}|Q|\right\}.
\end{align*}
Since there is one $i$ such that $|Q_{ii}|<\frac  {\sqrt 3} {3}|Q|$,
then $\{U_i\}_{i=1}^3$ is an open cover of $S$ and
let $\{\xi_i\}_{i=1}^3$ be a smooth partition of unity subordinate to the open cover such
that  $\sum_{i=1}^3\xi_i=1$  and $0\leq \xi_i\leq 1$ in $S$, $\xi_i\in C_0^{\infty}(U_i)$ and $\xi_i=1$ in $V_i$, where $V_i$ is an open subset of $U_i$ and $\{V_i\}_{i=1}^3$ is also an open cover of $S$.
Then for each $Q\in S_0$,  we have
								\begin{align*}\alabel{portion for V}
								V(\tilde Q,\tilde\na \tilde Q)=(\xi_1(\frac Q{|Q|}) +\xi_2( \frac Q{|Q|}))V(R_1Q, R_1\na  QR_1^T) +\xi_3(\frac Q{|Q|}) V(R_2Q, R_2\na  QR_2^T).
								\end{align*}
								
When $Q\in S_*$, $Q=s_+ (u\otimes u-\frac 13 I)$  with $u\in S^2$.								Without of generality, we only consider the case that $\frac Q{|Q|}\in U_1$; i.e.  $|Q_{11}|<\frac  {\sqrt 3} {3}|Q|$.
Noting that $|Q|=\sqrt{\frac 2 3}s_+|u|^2$, we have
\[|u|^2-u_1^2=\sqrt{\frac32}s^{-1}_+|Q|-(s^{-1}_+Q_{11}+\frac13)=s^{-1}_+\(\sqrt{\frac 23}|Q|-_+Q_{11}\)>\frac {\sqrt 2-1} {s_+}|Q_{11}|.\]
Since $u\in S^2$, it follows from  \eqref{3.14} that
				\begin{align*}\alabel{div u into Q}
				&I_1:=(\div u)^2=\sum_i (u_i\div u)^2=\sum_i (s_+^{-1}(s_+^{-1}Q_{ij}+\frac13\delta_{ij})(\na\cdot Q_j))^2.
\end{align*}	
Let $Q_i$  be the $i$-th column of the $Q$ matrix. One can verify from \eqref {EV-form} that
		$$(\curl u)_i = \sum_js_+^{-1}u_j(\curl Q_{j})_i.$$Then we find			
			\begin{align*}\alabel{u into Q I2}
I_2:=	&|u\times \curl u|^2=|u\times (s_+^{-1}u_j(\curl Q_{j}))|^2=\left|\sum_js_+^{-1}(s_+^{-1}Q+\frac13I)_j\times\curl Q_{j}\right|^2.\end{align*}

Using \eqref{3.14} again, we rewrite the third and fourth terms of \eqref{EV-form} as
				\begin{align*}
				&I_3:=\sum_i u_i^2\(\na_1u_1-\na_2u_2-\na_3u_3-\frac{2u_1u_2\nabla_2u_1+2u_1u_3\nabla_3 u_1} {|u|^2- u_1^2}\)^2
				\\
				&=\sum_{i,j}s_+^{-4}(Q_{ij}+\sqrt{\frac16}|Q|\delta_{ij})^2\(\na_1Q_{1j}-\na_2Q_{2j}-\na_3Q_{3j} -\frac{2Q_{12}\na_2Q_{1j}-2Q_{13}\na_3Q_{1j}}{\sqrt{\frac 23} |Q|-Q_{11}} \)^2,
				\\
				&I_4:=\(\frac {|u|^2}{|u|^2-u_1^2}  (u_3\na_2u_1-u_2\na_3u_1)\)^2
				\\
				&=\frac {3|Q|^2} {2s_+^{4}}\sum_{i} \(\frac{ (Q_{i3}+\sqrt{\frac16}|Q|\delta_{i3})\na_2Q_{1i}-(Q_{i2}+\sqrt{\frac16}|Q|\delta_{i2}) \na_3Q_{1i}}{ \sqrt{\frac 23} |Q|-Q_{11}} \)^2.
				\end{align*}
				 We rewrite the fifth term  of \eqref{EV-form} into
				\begin{align*}
				&I_5:=\( (u_3\na_1u_2-u_2\na_1u_3)+u_1(\na_2u_3- \na_3u_2)+\frac {u_1^2(u_3\nabla_2u_1-u_2\nabla_3 u_1) }{|u|^2-u^2_1}\)^2
				\\
				&=\sum_{i}s_+^{-4}\((Q_{i3}+\sqrt{\frac16}|Q|\delta_{i3})\na_1Q_{2i}-(Q_{i2}+\sqrt{\frac16}|Q|\delta_{i2}) \na_1Q_{3i}\right.
				\\
				&\left.+(Q_{i1}+\sqrt{\frac16}|Q|\delta_{i1})\Big((\na_2Q_{3i}- \na_3Q_{2i})+\frac{Q_{13}\na_2Q_{1i}-Q_{12}\na_3Q_{1i}}{\sqrt{\frac 23} |Q|-Q_{11}} \Big)\)^2.
				\end{align*}
				Finally, we write the last term in (3.8) as
				\begin{align*}
				I_6  =&\sum_{i}s_+^{-4}\left[\left.\frac{(1+\sign{L_3})\sqrt{\frac 32} |Q|}{\sqrt{\frac 23} |Q|-Q_{11}} \right.\right((Q_{i3}+\sqrt{\frac16}|Q|\delta_{i3})\na_2Q_{1i}
								\\
								&-(Q_{i2}+\sqrt{\frac16}|Q|\delta_{i2}) \na_3Q_{1i}\left)+\sign{L_3}\((Q_{i3}+\sqrt{\frac16}\right.|Q|\delta_{i3})\na_1Q_{2i}\right.
								\\
								&\left.\left.-(Q_{i2}+\sqrt{\frac16}|Q|\delta_{i2}) \na_1Q_{3i}+(Q_{i1}+\sqrt{\frac16}|Q|\delta_{i1})(\na_2Q_{3i}- \na_3Q_{2i})\)\right]^2.
				\end{align*}
				Substituting the identities of $I_1$,..., $I_6$ into the equation \eqref{EV-form}, we have
				\begin{align*}\alabel{eq V(Q) form}
				\quad V(Q,\na Q) =&\frac{L_1+  L_2+ \frac 12 L_3- \frac{s_+}{3}L_4-\alpha}{4}I_1 +\frac{(L_1+\frac{1 }{2}L_2+\frac{1}{2}L_3+\frac{2s_+}{3}L_4-\alpha) }{2s_+^{2}}I_2
\\
& +\frac{L_1-\frac{s_+}{3}L_4+\frac{1}{2}L_3-\alpha}{2}I_3+\frac{(L_1-\frac{s_+}{3}L_4-\frac{1}{2}|L_3|- \alpha)}{2}I_4\\
&+\frac{(L_1-\frac{s_+}{3}L_4-\frac{1}{2}|L_3|- \alpha)}{2}I_5+\frac{|L_3|}{2} I_6.
				\end{align*}
				
				Repeat this process for the remaining cases for $\xi_i$ in \eqref{portion for V} and use the relation \eqref{eq alpha} for $\alpha$. We see that
					\begin{align*}\alabel{Full W(Q,na Q)}
					f_E(Q,\nabla Q)&=\frac   \alpha 2   |\na Q|^2+V(Q,\na Q),
						\end{align*}
						where one can find an explicit form of $V(Q,\na Q)$ that is a sum of square terms and quadratic in  $\na Q$ satisfying
						\begin{align*}
						V(  Q,  \na Q)\leq C(1+|Q|^2) |\nabla Q|^2, \quad |V_Q(  Q,  \na Q)|\leq C(1+|Q|)|\nabla Q|^2.
						\end{align*}
						This completes a proof.
		\end{proof}
As a consequence of Theorem \ref {Theorem 2}, we give a proof of Corollary 1.
 \begin{proof}
We first note that
\[ (u\cdot \curl u)^2=(s_+^{-1}u_iu_j(\curl Q_j)_i)^2=\(\sum_{i,j} s_+^{-1}(s_+^{-1}Q_{ij}-\frac13\delta_{ij})(\curl Q_j)_i\)^2.\]
Using \eqref{Relation}, \eqref{div u into Q} and \eqref{u into Q I2}, we   write \eqref{eq Giaq} as
			\begin{align*}
		& 2W(Q,\na Q) =  \tilde \alpha |\na Q|^2+2V(Q,\na Q)
		\\
		=& s_+^{-2} \tilde \alpha |\na Q|^2+(k_1-\tilde \alpha)\sum_k\left(s_+^{-1}\sum_{i,j}(s_+^{-1}Q_{kj}+\frac13 \delta_{kj})\nabla_iQ_{ij}\right)^2
		\\
		&+  (k_2-\tilde \alpha)\( \sum_{i,j}s_+^{-1}(s_+^{-1}Q_{ij}+\frac13 \delta_{ij})(\curl Q_j)_i\)^2
		\\
		&+(k_3-\tilde \alpha)\left|\sum_{j}s_+^{-1}(s_+^{-1}Q+\frac 13I)_j\times\curl Q_{j}\right|^2
		\\
		=& \alpha |\na Q|^2+(2 L_1+L_2+L_3-\frac{2s_+}{3}L_4-\alpha)\sum_k\left(\sum_{i,j}(s_+^{-1}Q_{kj}+\frac13 \delta_{kj})\nabla_iQ_{ij}\right)^2
		\\
		&+  (2L_1-\frac{2s_+}{3}L_4-\alpha)\( \sum_{i,j}(s_+^{-1}Q_{ij}+\frac13 \delta_{ij})(\curl Q_j)_i\)^2
		\\
		&+(2L_1+L_2+L_3+\frac{4s_+}{3}L_4-\alpha)\left|\sum_{j}(s_+^{-1}Q+\frac 13I)_j\times\curl Q_{j}\right|^2 .
		\end{align*}
	\end{proof}
\section{Proof of Theorem 3}
\begin{lemma}
	If $Q$ is a minimizer of $\tilde E_{L}$ in $W^{1,2}_{Q_0}(\Omega; S_0)$, it   satisfies
	\begin{align*}
	&-\tilde\alpha  \Delta Q_{ij}-\frac 1 2  \na_k (V_{p^k_{ij}}+ V_{p^k_{ji}})+\frac 1 3 \delta_{ij}\sum_l \nabla_k V_{p^k_{ll}}  +\frac 1 2 (V_{Q_{ij}}  +V_{Q_{ji}})-   \frac 1 3 \delta_{ij}\sum_l V_{Q_{ll}}  \\& + \frac 1 L \(-aQ_{ij}- b  (Q_{ik}Q_{kj}-\frac 13\delta_{ij}\tr(Q^2) )+cQ_{ij}\tr(Q^2)\)=0
	\end{align*}
	in the weak sense.	
\end{lemma}

\begin{proof}
For any test function $\phi\in C^\infty_0(\Omega; S_0)$, consider $Q_t:=Q+t\phi$ for $t\in \R$.
Then for all $\phi\in C^\infty_0(\Omega; S_0)$, we calculate
\begin{align*}
&\left.\int_\Omega \frac{\d}{\d t}\(\tilde f_{E,1}(Q_t,\na Q_t)+\frac 1L\tilde f_B(Q_t)\)\right|_{t=0}\d x
\\
=&\int_\Omega \tilde\alpha  \frac{\p Q_{ij}}{\p x_k}\frac{\p \phi_{ij}}{\p x_k}+ V_{p^k_{ij}} \frac{\p \phi_{ij}}{\p x_k}+V_{Q_{ij}} \phi_{ij}\d x
\\
&+ \frac 1 L\int_\Omega  -aQ_{ij}\phi_{ij}-  b  Q_{ik}Q_{kj}\phi_{ij} +c(Q_{ij}\tr(Q^2)\phi_{ij})\d x
\\
=&\int_\Omega\( -\tilde\alpha  \Delta Q_{ij} -\frac 1 2\frac{\p }{\p x_k}(V_{p^k_{ij}} + V_{p^k_{ji}})
+\frac 12 (V_{Q_{ij}} +V_{Q_{ji}}) \)\phi_{ij}\d x \\
&+ \frac 1 L\int_\Omega  \(-aQ_{ij}- b  Q_{ik}Q_{kj} +cQ_{ij}\tr(Q^2)\)\phi_{ij}\d x\\
=& \int_\Omega\( -\tilde\alpha  \Delta Q_{ij} -\frac 1 2 \nabla_k(V_{p^k_{ij}}+  V_{p^k_{ji}}) -\frac 1 3 \delta_{ij}\sum_l \nabla_k V_{p^k_{ll}}
  \)\phi_{ij}\d x
\\
& +\int_\Omega\(\frac 12 (V_{Q_{ij}} +V_{Q_{ji}}) - \frac 1 3 \delta_{ij}\sum_l V_{Q_{ll}}  \)\phi_{ij}\d x
\\
& + \frac 1 L\int_\Omega  \(-aQ_{ij}- b  (Q_{ik}Q_{kj}-\frac 1 3 \delta_{ij}\tr(Q^2) ) +cQ_{ij}\tr(Q^2)\)\phi_{ij}\d x
=  0,
\end{align*}
where we used the fact that $\phi$ is traceless.
 This proves our claim.
\end{proof}

Then we will show that
\begin{lemma}\label{lem uniform bound}
Let $Q_L$ be a weak solution to the equation (\ref{MDEL}) with the boundary value $Q_0\in W^{1,2}(\Omega ; S_*)$. Then,  $|Q_L|\leq M+1$ for a sufficient large $M$.
\end{lemma}
\begin{proof} Recall from the definition of $\tilde f_E$ in (\ref{MD}) that for a $Q\in S_0$ with $|Q|\geq M+1$,
\[\tilde f_E(Q, \nabla Q)=\frac {\tilde \alpha} 2|\nabla Q|^2.\]
Similarly to one in \cite {CHH}, choose  a test function $\phi =Q (1-\min\{1, \frac {M+1}{|Q|}\})$. Multiplying (\ref{MDEL})   by the test function $\phi$, we have
\begin{align*}
	&\tilde\alpha  \int_{|Q|\geq M+1}|\nabla Q|^2 (1- \frac {M+1}{|Q|})-(M+1) Q_{ij} \na_k Q_{ij} \nabla_k  \frac 1{|Q|})\,dx
 \\& + \frac 1 L \int_{|Q|\geq M+1} \(-a|Q|^2- b  Q_{ik}Q_{kj} Q_{ij} +c|Q|^4\)
 (1- \frac {M+1}{|Q|}) \,dx=0.
	\end{align*}
Note the fact that $  \nabla_k |Q|^2=2Q_{ij} \na_k Q_{ij} $. The above second term is nonnegative.
For a sufficiently large $M>0$, third term also is positive. This implies that the set $\{|Q|\geq M+1\}$ is empty; i.e.
$|Q|\leq M+1$ a.e. in $\Omega$.
\end{proof}

\begin{lemma}
\label{lem Legendre-Hadamard condition}
For any $Q_*\in S_*$, the Hessian of the bulk density $\tilde f_B(Q_*)$ is positive definite for a uniform constant;  i.e.
for any $\xi \in S_0$, we have
\begin{align*}\alabel{eq positive definite}
\p_{Q_{ij}}\p_{Q_{kl}} f_B(Q_*)\xi_{ij}\xi_{kl}\geq \lambda |\xi|^2,
\end{align*}
where $\lambda =\min\{a,s_+ b \}>0$.
\end{lemma}
\begin{proof}
Recall the fact that the bulk density $f_B$ is rotational invariant. For any tensor $Q \in S_*$, there exists a rotation $R=R(Q)\in SO(3)$ such that we can rotate $Q$ to its diagonal form $\tilde Q$ with elements $(\frac{-s_+}{3},\frac{-s_+}{3},\frac{2s_+}{3})$ and
\[\tilde Q_{ij}=R_{ip}Q_{pq}R_{jq}.\]
Using the chain rule, we derive
\begin{align*}
&\p_{Q_{mn}}\p_{Q_{kl}}   f_B(Q)\xi_{mn}\xi_{kl}=\p_{Q_{mn}}\(\frac{\p f_B(\tilde Q)}{\p \tilde Q_{ij}}\frac{\p\tilde Q_{ij}}{\p Q_{kl}}\)\xi_{mn}\xi_{kl}
\\
=&\frac{\p^2 f_B(\tilde Q)}{\p \tilde Q_{ij}\p \tilde Q_{\tilde i\tilde j}}\frac{\p \tilde Q_{ij}}{\p Q_{kl}}\frac{\p \tilde Q_{\tilde i\tilde j}}{\p Q_{mn}}\xi_{mn}\xi_{kl}
\\
=&  \frac{\p^2 f_B(\tilde Q)}{\p \tilde Q_{ij}\p \tilde Q_{\tilde i\tilde j}}\frac{\p (R_{ip}Q_{pq}R_{jq})}{\p Q_{kl}}\frac{\p(R_{\tilde i\tilde p}Q_{\tilde p\tilde q}R_{\tilde j\tilde q})}{\p Q_{mn}}\xi_{mn}\xi_{kl}
\\
=& \frac{\p^2 f_B(\tilde Q)}{\p \tilde Q_{ij}\p \tilde Q_{\tilde i\tilde j}}R_{ik} R_{jl}R_{ \tilde im} R_{\tilde jn}\xi_{mn}\xi_{kl}=\frac{\p^2 f_B} {\p \tilde Q_{ij}\p \tilde Q_{\tilde i\tilde j}}(\tilde Q)\tilde \xi_{ij} \tilde \xi_{\tilde i\tilde j},
\end{align*}
where $\tilde \xi_{ij}=R_{ik} \xi_{kl}R_{jl}$ and $\tilde \xi_{\tilde i\tilde j}=R_{ \tilde im} \xi_{mn}R_{\tilde jn}$.

We calculate  the first derivative of $f_B(\tilde Q)$
\begin{align*}
\p_{\tilde Q_{ij}}f_B(\tilde Q) =\(-a \tilde Q_{ij}-b\sum_k\tilde Q_{jk}\tilde Q_{ki}+c\tilde Q_{ij}|\tilde Q|^2\).
\end{align*}
Then the second derivative of $f_B(\tilde Q)$ is
\begin{align*}
& \p_{\tilde Q_{\tilde i\tilde j}}\p_{\tilde Q_{ij}}f_B(\tilde Q)= -a\delta_{i\tilde i}\delta_{j\tilde j}
-b(\delta_{\tilde ij}\tilde Q_{\tilde j i}+\delta_{\tilde ji}\tilde Q_{j\tilde i})+c(\delta_{i\tilde i}\delta_{j\tilde j}|\tilde Q|^2+2\tilde Q_{ij}\tilde Q_{\tilde i\tilde j}) .
\end{align*}
For the case of $i=j=\tilde i=\tilde j$, using the equality $\frac 23 cs^2_+=\frac 13bs_++a $ (c.f. \cite{MZ}), we find
\begin{align*}
\p_{\tilde Q_{ii}}\p_{\tilde Q_{ii}}f_B(\tilde Q) =&-a -2\tilde Q_{ii}b+( |\tilde Q|^2+2\tilde Q_{ii}^2)c=-(2\tilde Q_{ii}-\frac{s_+}{3})b+2\tilde Q_{ii}^2c.
\end{align*}
Then, at $\tilde Q=\tilde Q_*$,  we have
\begin{align}
\p_{\tilde Q_{11}}\p_{\tilde Q_{11}}f_B(\tilde Q)  =&\(s_+b+\frac{2s_+^2}{9}c\)=\frac 13a+\frac{10s_+}{9}b,
\\
\p_{\tilde Q_{22}}\p_{\tilde Q_{22}}f_B(\tilde Q) =&\frac 13a+\frac{10s_+}{9}b,
\\
\p_{\tilde Q_{33}}\p_{\tilde Q_{33}}f_B(\tilde Q)=&-s_+ b+\frac{8s_+}{9}c=\frac43a-\frac{5s_+}{9}b.
\end{align}
For the case of $i=j\neq\tilde i=\tilde j$,  three terms at $\tilde Q=\tilde Q_*$ are
\begin{align}
2\p_{\tilde Q_{11}}\p_{\tilde Q_{22}}f_B(\tilde Q) =&4\tilde Q_{11}\tilde Q_{22}c=\frac{4s_+^2}{9}c=\frac23a+\frac{2s_+}{9}b,
\\
2\p_{\tilde Q_{11}}\p_{\tilde Q_{33}}f_B(\tilde Q) =&4\tilde Q_{11}\tilde Q_{33}c=-\frac{8s_+^2}{9}c=-\(\frac43a+\frac{4s_+}{9}b\),
\\
2\p_{\tilde Q_{22}}\p_{\tilde Q_{33}}f_B(\tilde Q) =&4\tilde Q_{22}\tilde Q_{33}c=-\frac{8s_+^2}{9}c=-\(\frac43a+\frac{4s_+}{9}b\).
\end{align}
For the remaining case,  that is $i\neq j$ or $ \tilde i\neq \tilde j$, we have  at $\tilde Q=\tilde Q_*$
\begin{align*}\alabel{Second derivative case 3}
&\(\sum_{i\neq j}\sum_{\tilde i,\tilde j}+\sum_ {\tilde i\neq \tilde j}\sum_{i,j}\)\p_{\tilde Q_{\tilde i\tilde j}}\p_{\tilde Q_{ij}}f_B(\tilde Q)\tilde \xi_{\tilde i\tilde j}\tilde \xi_{ij}
\\=& \(\sum_{i\neq j}\sum_{\tilde i,\tilde j}+\sum_ {\tilde i\neq \tilde j}\sum_{i,j}\)\(c(\delta_{i\tilde i}\delta_{j\tilde j}|\tilde Q|^2+2\tilde Q_{ij}\tilde Q_{\tilde i\tilde j})-b(\delta_{\tilde ij}\tilde Q_{\tilde j i}+\delta_{\tilde ji}\tilde Q_{j\tilde i})\)\tilde \xi_{\tilde i\tilde j}\tilde \xi_{ij}
\\
=&\sum_{i\neq j}\(\frac{2s_+^2}{3}c-(\tilde Q_{ii}+\tilde Q_{jj})b\)|\tilde \xi_{ij}|^2\geq\sum_{i\neq j}a|\tilde \xi_{ij}|^2.
\end{align*}
In conclusion, using the fact that $\tilde \xi_{33}=-(\tilde\xi_{11}+\tilde\xi_{22})$ we have at $\tilde Q=\tilde Q_*$
\begin{align*}
& \p_{Q_{mn}}\p_{Q_{kl}}   f_B(Q)\xi_{mn}\xi_{kl}=\p_{\tilde Q_{\tilde i\tilde j}}\p_{\tilde Q_{ij}}f_B(\tilde Q)\tilde \xi_{\tilde i\tilde j}\tilde \xi_{ij}
\\
\geq&\(\frac 13a+\frac{10s_+}{9}b\)(\tilde \xi_{11}^2+\tilde \xi_{22}^2)+\(\frac23a+\frac{2s_+}{9}b\)\tilde \xi_{11}\tilde \xi_{22}
\\
&+\(\frac43a-\frac{5s_+}{9}b \)\tilde \xi_{33}^2-\(\frac43a+\frac{4s_+}{9}b\)\tilde \xi_{33}(\tilde \xi_{11}+\tilde \xi_{22})+\sum_{i\neq j}a| \tilde \xi_{ij}|^2
\\
=&  bs_+(\tilde \xi_{11}^2+\tilde \xi_{22}^2) + 3a\tilde \xi_{33}^2 +\sum_{i\neq j}a|\tilde \xi_{ij}|^2\geq \lambda|\xi|^2
\end{align*}
with $\lambda =\min\{ a,s_+b\}>0$.
\end{proof}

Now we give a proof of Theorem \ref{Theorem 3}.
\begin{proof}For each $L>0$, let $Q_L$ be a weak solution to the equation (\ref{MDEL}) with boundary value $Q_0\in W^{1,2}(\Omega, S_*)$.
		By Lemma \ref{lem uniform bound}, $Q_L$ is uniformly bounded in $\Omega$.

For a small $\delta >0$, let $S_{\delta}$ be a neighborhood of $S_*$ defined by
\[S_{\delta}:=\{Q\in S_0:\quad \mbox{dist} (Q, S_*)\leq \delta\}.\]
There is a smooth projection $\pi$ from $S_{2\delta}$ to   $S_*$ (see \cite{CS}). For each $\delta>0$, define a set
		\[\Sigma_{\delta}=S_0\backslash S_{\delta}=\{Q\in S_0: \mbox{dist}(Q, S_*)\geq \delta\}.\]
For each $Q\in \Sigma_{\delta}$, we have $\pi (Q)\in S_*$;
i.e. $ \pi (Q) =s_+\(u\otimes u-\frac 1 3 I\)$ with $u\in S^2$.

For a   test function $\phi\in C^\infty_0(\Omega; \R^3)$ and  a small $t\in\R$, set  $u_{t}:= \frac{u+t\phi}{|u+t\phi|}$.
  Then we define
	\begin{align}\label{eq projection variation}
	 \pi (Q)_t :=  \ s_+\(u_{t}\otimes u_{t}-\frac 1 3 I\)\in S_*.
		\end{align}

By the Taylor expansion for $\tilde f_B(\pi (Q_L)_t)$ at $Q_L\in S_{\delta}$,  we derive
		\begin{align*}\alabel{eq Taylor}
				\frac{\tilde f_B(\pi (Q_L)_t)}{L} = & \frac{\tilde f_B(Q_L)}{L}+ \frac1L\na_{Q_{ij}} f_B(Q_L)(\pi (Q_L)_t -Q_L)_{ij}
				\\
				+ &\frac1{2L}\na^2_{Q_{ij}Q_{kl}}  f_B\(Q_{\tau}\)(\pi (Q_L)_t-Q_L)_{ij}( \pi (Q_L)_t-Q_L)_{kl},
		\end{align*}
where $Q_\tau:=(1-\tau) \pi (Q_L)_t+\tau Q_L$ for some $\tau\in[0,1]$.

Since  $\pi (Q_L)_t\in S_*$, it implies that $\tilde f_B(\pi (Q_L)_t)=0$.  Noting that the function $\tilde f_B(Q)$ is smooth in $Q$, for any  $\varepsilon_1 >0$, there is a $\delta_1>0$ such  that for any two $Q_1,  Q_2$ bounded by $M+1$ with $|Q_1-Q_2|\leq \delta_1$, we have
\[|\na^2_{Q_{ij}Q_{kl}} f_B\(Q_1\)-\na^2_{Q_{ij}Q_{kl}} f_B\( Q_2 \)|\leq \varepsilon_1.\]
For sufficiently small $L$ and $t$  such that $|Q_{L,t}-Q_*|<\frac{1-\tau}{2}\delta_1$ and $\delta=\frac{\tau}{2}\delta_1$,  we have $|Q_\tau-Q_*|<\delta_1$. By
choosing $\varepsilon_1$ sufficiently small and applying Lemma \ref{lem Legendre-Hadamard condition}, we obtain
\[\na^2_{Q_{ij}Q_{kl}} f_B\(Q_\tau\)( \pi (Q_L)_t-Q_L)_{ij}( \pi (Q_L)_t-Q_L)_{kl} \geq\frac{\lambda}{2}| \pi (Q_L)_t-Q_L|^2.\]
For each $L$, we define a subdomain by
\[\Omega_{L,\delta} =\{x\in \Omega : \, Q_L(x)\in S_{\delta}\}.\]
For a sufficiently small $\delta$ and $t$, we have
\begin{align*}\alabel{eq 2nd order fb}
	&\int_{\Omega_{L,2\delta}}\frac 1L\na^2_{Q_{ij}Q_{kl}}  f_B\(Q_{\tau}\)(\pi (Q_L)_t -Q_L)_{ij}(\pi (Q_L)_t -Q_L)_{kl} \d x \\
\geq&\frac 1L\int_{\Omega_{L,2\delta}}\frac{\lambda}{2}| \pi (Q_L)_t-Q_L|^2\d x.
\end{align*}
Then it follows from \eqref{eq Taylor} that
\begin{align*}\alabel{eq f_B out of Sigma_L}
 \int_{\Omega_{L,2\delta}} \frac1L\na_{Q_{ij}} f_B(Q_L)(\pi (Q_L)_t -Q_L)_{ij} \d x\leq 0.
\end{align*}
In order to  extend \eqref{eq f_B out of Sigma_L} to $\Omega$, we  define
\begin{align*}\alabel{eq extension}
\hat  Q_{L,t}:=\begin{cases}
\pi(Q_L)_t,&\mbox{ for }Q_L\in S_{\delta}
\\
\frac{|Q_L-\pi(Q_L)|^2}{\delta^2}\pi(Q_L)_t+\frac{\delta^2-|Q_L-\pi(Q_L)|^2}{\delta^2} Q_{*,t},&\mbox{ for }Q_L\in \Sigma_{\delta}\backslash \Sigma_{2\delta}
\\
Q_{*,t},&\mbox{ for } Q_L\in \Sigma_{2\delta}.
\end{cases}
\end{align*}
It can be checked that $\hat  Q_{L,t}\in W^{1,2}_{Q_0}(\Omega ; S_0)$. Then
\begin{align*}\alabel{TB}
\hat  Q_{L,t}-Q_{*,t}=\begin{cases}
\pi(Q_L)_t- Q_{*,t},&\mbox{ for }Q_L\in S_{\delta}
\\
\frac{|Q_L-\pi(Q_L)|^2}{\delta^2} (\pi(Q_L)_t- Q_{*,t}),&\mbox{ for }  Q_L\in \Sigma_{\delta}\backslash \Sigma_{2\delta}
\\
0,&\mbox{ for } Q_L\in \Sigma_{2\delta}.
\end{cases}
\end{align*}

On the other hand, there is a uniform bound for $\tilde f_B(Q_L(x)) \geq C(\delta)>0,\forall x\in \Omega \backslash \Omega_{L,\delta}$. Using Lemma \ref{lem uniform bound} we observe that
\begin{align*}\alabel{eq f_B in Sigma_L}
&\int_{\Omega \backslash \Omega_{L,\delta} }\frac1L\na_{Q_{ij}} f_B(Q_L)(\hat Q_{L,t}-Q_L)_{ij}\d x
\\
=&\int_{\Omega_{L,2\delta} \backslash \Omega_{L,\delta} }\frac1L\na_{Q_{ij}} f_B(Q_L)\left [\frac{|Q_L-\pi(Q_L)|^2}{\delta^2}(\pi(Q_L)_t-Q_{*,t})+(Q_{*,t}-Q_L)\right]_{ij}\d x
\\
&+\int_{\Omega \backslash \Omega_{L,2\delta} }\frac1L\na_{Q_{ij}} f_B(Q_L)(Q_{*,t}-Q_L)_{ij}\d x \\
\leq& C \frac{| \Omega \backslash \Omega_{L,\delta}|}{L}  \leq  \frac C {C(\delta)}\int_{\Omega \backslash \Omega_{L,\delta} }\frac{\tilde f_B(Q_L) }L \d x.
\end{align*}
By the assumption in Theorem 3, we have
\begin{align*}\alabel{eq na f_B}
\lim_{L\to 0}\int_{\Omega}\frac1L\na_{Q_{ij}} f_B(Q_L)(\hat Q_{L,t}-Q_L)_{ij}\d x\leq 0.
\end{align*}
Multiplying (\ref{MDEL}) by $(\hat Q_{L,t}-Q_L)$, integrating by parts and using \eqref{eq na f_B} yield
\begin{align*}\alabel{eq sub of na f_B}
&\lim_{L\to0}\int_{\Omega }\(\alpha \na_kQ_{L,ij} +\tilde V_{p^k_{ij}}(Q_L,\na Q_L)-\tilde V_{Q_{ij}}(Q_L,\na Q_L)\)\na_k (\hat Q_{L,t}-Q_L)_{ij}\d x\geq 0.
\end{align*}
Here we used the fact that $Q_{L,t}-Q_L$ is symmetric and traceless.

In order to pass a limit, we claim that $\hat  Q_{L,t}\to Q_{*,t}$ strongly in $W^{1,2}_{Q_0}(\Omega ; S_0)$.

\noindent In fact,  it follows from \eqref{TB} that
\begin{align*}\alabel{Q hat convergence}
&\int_{\Omega} |\na  (\hat  Q_{L,t}-Q_{*,t})|^2 \,dx= \int_{ \Omega_{L,2\delta}}  |\na  (\hat  Q_{L,t}-Q_{*,t})|^2 \,dx
\\
= & \int_{\Omega_{L,\delta}} |\na  ( \hat Q_{L,t}-Q_{*,t})|^2  \,dx
+\int_{ \Omega_{L,2\delta} \backslash  \Omega_{L,\delta} } \left |\na \(  \frac{|Q_L-\pi(Q_L)|^2}{\delta^2} (\pi(Q_L)_t- Q_{*,t})\)\right |^2  \,dx
\\
\leq&  \int_{\Omega_{L,\delta}}  |\na  ( \pi (Q_L)_t-\pi (Q_{*})_t)|^2  \,dx +
 C \int_{ \Omega_{L,2\delta} \backslash  \Omega_{L,\delta} }|\na  ( \pi (Q_L)_t-\pi (Q_{*})_t)|^2  \,dx
\\
&+C \int_{ \Omega_{L,2\delta} \backslash  \Omega_{L,\delta} }\frac {|\pi (Q_L)_t-Q_{*,t}|^2}{\delta^4}\( |\nabla ( Q_L-Q_*)|^2+ |\nabla ( \pi (Q_*)-\pi (Q_L))|^2\)\,dx.
\end{align*}
Note that
\begin{align*}
& \pi (Q_L)-\pi (Q_{*})=\na_{Q}\pi (Q_{\xi}) (Q_L-Q_{*}),\\
&\pi (Q_L)_t-\pi (Q_{*})_t=\na_{Q}\pi(Q_{\xi})_t (Q_L-Q_{*}).
\end{align*}
When $Q_L$ approaches to $Q_*$, $\na_{Q}\pi (Q_{\xi})$ is close to the identity map $I$ and $\na_{Q}\pi(Q_{\xi})_t$ for small $t$.  Therefore
\begin{align*}
& |\na (\pi (Q_L) -\pi (Q_{*}))  |\leq C|\nabla (Q_L-Q_{*})| +C |\nabla Q_{\xi}| |Q_L-Q_{*}|.
\end{align*}
As $Q_L\to Q_*$, the term $\pi(Q_L)_t$ is close to $\pi(Q_*)_t$ and $\na_Q\pi(Q_\xi)_t$ is close to the identity map for small $t$. Note that $\na^2_{QQ}\pi(Q_\xi)_t$ is bounded. Then
\begin{align*}
|\na  (\pi (Q_L)_t-\pi (Q_{*})_t)|&\leq |\na_Q\pi(Q_\xi)_t\na  (Q_L -Q_{*})|+|\na^2_{QQ}\pi(Q_\xi)_t||\na Q_\xi||Q_L -Q_{*}|
\\
&  \leq C|\nabla (Q_L-Q_{*})| + C|\nabla Q_{\xi}| |Q_L-Q_{*}|.
\end{align*}
Then the inequality \eqref{Q hat convergence} reads as
\begin{align*}
&\int_{\Omega} |\na  (\hat  Q_{L,t}-Q_{*,t})|^2 \,dx
\\
\leq& C\int_{ \Omega_{L,2\delta} }|\nabla (Q_L-Q_{*})|^2   +(|\nabla Q_L|^2 +|\nabla Q_*|^2)|Q_L-Q_{*}|^2\,dx
\\
\leq& C\int_{ \Omega }|\nabla (Q_L-Q_{*})|^2\d x  +C\(\int_{ \Omega\backslash \Sigma_\varepsilon}+\int_{\Sigma_\varepsilon }\)|\nabla Q_*|^2|Q_L-Q_{*}|^2\,dx.
\end{align*}
Here we employ Egoroff's theorem; i.e. for all $\varepsilon>0$, there exists a measurable subset $\Sigma_{\varepsilon} \subset \Omega$ such that
\begin{equation}\label{Egoroff}
| \Sigma_\varepsilon|\leq \varepsilon \mbox{ and }Q_L\to Q_*  \mbox{ uniformly on }\Omega\backslash\Sigma_{\varepsilon}.
\end{equation}As $\varepsilon\to 0$ and $L\to 0$, we prove the claim that $\hat  Q_{L,t}\to Q_{*,t}$ strongly in $W^{1,2}_{Q_0}(\Omega; S_0)$.

 We observe that
\begin{align*}
& \int_{\Omega }|\tilde V_{p^k_{ij}}(Q_L,\na Q_L)\na_k(\hat Q_{L,t}-Q_L)_{ij}-\tilde V_{p^k_{ij}}(Q_*,\na Q_*)\na_k(Q_{*,t}-Q_*)_{ij}|\d x
\\
&\leq \int_{\Omega}| \tilde V_{p^k_{ij}}(Q_L,\na Q_L)||(\na_kQ_{L,t} -\na_kQ_{*,t})_{ij}+(\na_kQ_*-\na_kQ_L)_{ij}|\d x
\\
&+\(\int_{ \Omega\backslash \Sigma_\varepsilon}+\int_{\Sigma_\varepsilon }\) |\tilde V_{p^k_{ij}}(Q_L,\na Q_L)\na_k(Q_{*,t}-Q_*)_{ij}-\tilde V_{p^k_{ij}}(Q_*,\na Q_*)\na_k(Q_{*,t}-Q_*)_{ij}|\d x
\end{align*}
and
\begin{align*}
& \int_{\Omega }|\tilde V_{Q_{ij}}(Q_L,\na Q_L) (\hat Q_{L,t}-Q_L)_{ij} -\tilde V_{Q_{ij}}(Q_*,\na Q_*) (Q_{*,t}-Q_*)_{ij}|\d x
\\
\leq &\(\int_{ \Omega\backslash \Sigma_\varepsilon}+\int_{\Sigma_\varepsilon }\) |\tilde V_{Q_{ij}}(Q_L,\na Q_L) (\hat Q_{L,t}-Q_L)_{ij} -\tilde V_{Q_{ij}}(Q_*,\na Q_L) (\hat Q_{L,t}-Q_L)_{ij}|\d x
\\
&+\int_{\Omega } |\tilde V_{Q_{ij}}(Q_*,\na Q_L) (\hat Q_{L,t}-Q_L)_{ij}-\tilde V_{Q_{ij}}(Q_*,\na Q_*) (\hat Q_{*,t}-Q_*)_{ij}|\d x.
\end{align*}
Using the uniform convergence of $Q_L$ in $\Omega\backslash\Sigma_\varepsilon$ and strong convergence of $\hat Q_{L,t}, Q_L$ in $W^{1,2}_{Q_0}(\Omega,S_0)$,   we derive
\begin{align*}
&\lim_{L\to0}\int_{\Omega }|\tilde V_{Q_{ij}}(Q_L,\na Q_L) (\hat Q_{L,t}-Q_L)_{ij} -\tilde V_{Q_{ij}}(Q_*,\na Q_*) (Q_{*,t}-Q_*)_{ij}|\d x=0,
\\
 &\lim_{L\to0}\int_{\Omega }| \tilde V_{p^k_{ij}}(Q_L,\na Q_L)\na_k(\hat Q_{L,t}-Q_L)_{ij} -\tilde V_{p^k_{ij}}(Q_*,\na Q_*)\na_k(Q_{*,t}-Q_*)_{ij}|\d x=0.
 \end{align*}
As $L\to 0$, the estimate \eqref{eq sub of na f_B} yields
\begin{align*}\alabel{eq L varepsilon to 0}
&\int_{\Omega }\( \alpha \na_kQ_{*,ij} +\tilde V_{p^k_{ij}}(Q_*,\na Q_*)\)\na_k(Q_{*,t}-Q_*)_{ij}\d x
\\
&+\int_{\Omega }\tilde V_{ij}(Q_*,\na Q_*) (Q_{*,t}-Q_*)_{ij}\d x\geq 0.
\end{align*}

For each $\eta \in C_0^{\infty}(\Omega,S_0)$, we define
\begin{align*}
\varphi_{ij} (Q,\eta):=&
(s_+^{-1}Q_{jl}+\frac 13 \delta_{jl})\eta_{il}+(s_+^{-1}Q_{il}+\frac 13 \delta_{il})\eta_{jl}\alabel{test function 1}
\\
&-2(s_+^{-1}Q_{ij}+\frac 13 \delta_{ij})(s_+^{-1}Q_{lm}+\frac 13 \delta_{lm})\eta_{lm}.
\end{align*}
For the estimate \eqref{eq L varepsilon to 0}, the limit in $t$ exists then using \eqref{V1} and \eqref{V2} that we have
\begin{align*}
\lim_{t\to0}\frac{(Q_t-Q_*)}{t}=\varphi(Q_*,\eta),\quad \lim_{t\to0}\na \frac{(Q_t-Q_*)}{t}=\na\varphi(Q_*,\eta).
\end{align*}
Dividing \eqref{eq L varepsilon to 0} by $t$ then as $t\to 0^+$ and $t\to 0^-$, we have
\begin{align*}
	\int_{\Omega  } \(\alpha \na_kQ_{*,ij}+V_{p^k_{ij}}(Q_*,\na Q_*) \)\na_k  \varphi_{ij}(Q_*,\eta)+V_{Q_{ij}} (Q_*,\na Q_*)\varphi_{ij}(Q_*,\eta)\d x= 0.
\end{align*}
Repeating same steps in \eqref{EL for one constant on S_*} and \eqref{EL for V on S_*}, we prove that $Q_*$ satisfies  \eqref{EL}.
\end{proof}

\medskip
\noindent{\bf Acknowledgements:} We would like to thank
Professor John Ball  for his interest and valuable comments. In particular, his beautiful talk on the Landau-de Gennes theory at the University of Queensland in January 2018  has inspired us to work at this problem.  We also wish to thank Professor Arghir Zarnescu for his valuable comments.  Part of the research was supported by the Australian Research Council
grant DP150101275.

\end{document}